\documentclass[12pt]{amsart}

\usepackage{amssymb}
\usepackage{amsthm}
\usepackage{amsmath}
\usepackage{amscd}

\usepackage{verbatim}
\usepackage[all]{xy}
\numberwithin{equation}{section}

\theoremstyle{plain}
\newtheorem{theorem}{Theorem}[section]
\newtheorem{corollary}[theorem]{Corollary}
\newtheorem{lemma}[theorem]{Lemma}
\newtheorem{proposition}[theorem]{Proposition}

\newtheorem{conjecture}[theorem]{Conjecture}
\theoremstyle{definition}
\newtheorem{definition}[theorem]{Definition}
\newtheorem{remark}[theorem]{Remark}
\newtheorem{example}[theorem]{Example}

\theoremstyle{remark}

\newcommand{\OO}{\mathcal O}
\newcommand{\A}{\mathbb{A}}
\newcommand{\R}{\mathbb{R}}

\newcommand{\Q}{\mathbb{Q}}
\newcommand{\Z}{\mathbb{Z}}

\newcommand{\C}{\mathbb{C}}

\renewcommand{\H}{\mathbb{H}}

\newcommand{\D}{\mathbb{D}}

\newcommand{\X}{\mathbb{X}}

\newcommand{\zxz}[4]{\begin{pmatrix} #1 & #2 \\ #3 & #4 \end{pmatrix}}
\newcommand{\abcd}{\zxz{a}{b}{c}{d}}

\newcommand{\kzxz}[4]{\left(\begin{smallmatrix} #1 & #2 \\ #3 & #4\end{smallmatrix}\right) }

\newcommand{\vol}{\operatorname{vol}}
\newcommand{\tr}{\operatorname{tr}}

\newcommand{\Spin}{\operatorname{Spin}}

\newcommand{\Mp}{\operatorname{Mp}}

\newcommand{\Aut}{\operatorname{Aut}}

\newcommand{\End}{\operatorname{End}}
\newcommand{\Iso}{\operatorname{Iso}}

\newcommand{\Pet}{\text{\rm Pet}}

\newcommand{\GL}{\operatorname{GL}}

\newcommand{\Pic}{\operatorname{Pic}}
\newcommand{\Div}{\operatorname{Div}}

\newcommand{\z}{\operatorname{Z}}

\newcommand{\diag}{\operatorname{diag}}
\newcommand{\ord}{\operatorname{ord}}
\newcommand{\kay}{k}
\newcommand{\Gspin}{\operatorname{GSpin}}

\newcommand{\ff}{\hbox{if }}
\newcommand{\SL}{\operatorname{SL}}

\newcommand{\B}{\mathbb B}

\newcommand{\CH}{\operatorname{CH}}
\newcommand{\GS}{\operatorname{GS}}

\newcommand{\Naive}{\operatorname{Naive}}
\newcommand{\Cusp}{\operatorname{Cusp}}
\newcommand{\MW}{\operatorname{MW}}

\begin{document}

\title{Arithmetic Siegel Weil formula on $X_{0}(N)$}

\date{\today}
\author[Tuoping Du]{Tuoping Du}
\address{Department of Mathematics, Northwest University, Xi'an, 710127 ,  P.R. China}
\email{dtp1982@163.com}
\author[Tonghai Yang]{Tonghai Yang}
\address{Department of mathematics, Wisconsin University}
\email{thyang@math.wisc.edu}

\begin{abstract} In this paper, we proved an arithmetic Siegel-Weil formula and the modularity of some arithmetic theta function on the modular curve $X_0(N)$ when $N$ is square free.
In the process, we also constructed some generalized Delta function for $\Gamma_0(N)$ and proved some explicit Kronecker limit formula for Eisenstein series on $X_0(N)$.

\end{abstract}

\dedicatory{}


\thanks{The first author is  partially supported by NSFC(No.11401470), and the second author is partially supported a NSF grant DMS-1500743.}

%\subjclass[2000]{14G35, 11F55, 11G18}

\maketitle

\section{introduction}\label{introduction}\pagenumbering{arabic}\setcounter{page}{1}

It is well-known  that there is a deep connection between the leading term of some analytic  functions and the arithmetics, such
as the class number formula, Birch and Swinnerton-Dyer conjecture, Block-Kato conjecture and the Siegel-Weil formula.  Little  is known or understood
about the possible connection between the second term of these functions and arithmetic although it started to change in this century.  The most famous one
is the Kronecker  limit formula:
$$
E(\tau, s) =\sum_{\gamma \in \Gamma_\infty \backslash \SL_2(\Z)} \Im(\gamma z)^s
          = 1 + \frac{1}6 (\log |\Delta(\tau) \Im(z)^6|) s + O(s^2).
$$
We refer to \cite{Siegel} for its proof and its  beautiful application to class numbers.   In 2004, Kulda, Rapoport and Yang
(\cite{KRYComp}) discovered another second term identity of some Eisenstein series of weight $3/2$---the so-called arithmetic Siegel-Weil
formula. Roughly speaking, they defined an arithmetic function---a generating function $\widehat{\phi}_{KRY}(\tau)$ of  a family of arithmetic divisors in a Shimura curve.  They proved that  its degree is the special value of some Eisenstein series $\mathcal E(\tau, s)$ (weight $3/2$) at $s=1$ and that arithmetic intersection  with the (normalized) metrized Hodge bundle on the Shimura curve is the derivative of the same Eisenstein series $\mathcal E(\tau, s)$ at $s=1$ (second term).
 This case is different from the Kronecker formula in two ways. Firstly, the leading term is already connected with arithmetics by the
Siegel-Weil formula and is non-trivial. Secondly, the second term (derivative) is found to be deeply related with the Gillet-Soul\'e  height paring on a Shimura curve. Its
analogue in $X_0(1)$ was worked out later by Kudla and Yang, and was reported in \cite{Yazagier}. In this case, the Eisenstein series is Zagier's famous Eisenstein series \cite{HZ} of weight $3/2$. In \cite{BF1}, Bruinier and Funke gave a different
proof of the main result of  \cite{Yazagier} using theta lifting. Colmez conjecture \cite{Col} can also be viewed as an second term of `CM' Hecke $L$-functions $L'(0, \chi)$ in terms of Faltings' height.  We should mention the  breakthrough formula of Zhiwei Yun and
Wei Zhang which  relates the $n$-th  central derivative of the $L$-function of an automorphic representation on $\GL_2$  over a function field and
height pairing of some cycles in middle dimension  on some Drinfeld space \cite{YunZhang}. We also mention the beautiful second term identity in
the Siegel-Weil formula (see for example \cite{GQT} and references there), although it has different flavor.

Later in the book \cite[Chapter 4]{KRYBook}, Kudla proved that the arithmetic theta function $\widehat{\phi}_{KRY}$ is modular. In this paper, we will prove  both the arithmetic Siegel-Weil formula and the modularity of  a similar arithmetic theta function in the case of modular curve $X_0(N)$ when  $N$ is square free. The complication comes mainly from the cusps, and we need to understand the behavior of Kudla's Green functions at cusps carefully. We give a complete description  of its behavior at cusps--which is totally new. It is an interesting and likely very challenging question to extend the analysis to high dimensional Shimura varieties of orthogonal type $(n, 2)$.  The metrized Hodge bundle has log singularity at cusps presents another complication.  The method in \cite{KRYComp} in computing the arithmetic intersection   does not seem to extend to this case easily. Instead, we will use theta lifting method following  \cite{BF1}.   In the process, we also obtain some explicit Kronecker limit formula for Eisenstein series  of weigh $0$ for $\Gamma_0(N)$, which should be of independent interest. In particular, we construct an  explicit modular form (denoted by $\Delta_N$), which  gives  a  rational section of Hodge bundle and  plays an essential role in  proving the arithmetic Siegel-Weil formula. After the arithmetic Siegel-Weil formula is proved, the modularity theorem  follows the same method of \cite[Chapter 4]{KRYBook} with a little modification.

Now we set up notations and describe the main results in a little more detail.

\begin{comment}

Briefly, Yang defined an arithmetic generating function
$$
\widehat\Phi_{KRY}(\tau) = \sum_{n \in \Z} \widehat{\mathcal Z}(n, v) q^n \in  \widehat{\CH}_\R^1(\mathcal X_0(4))[[q]]
$$
which is a modular form valued in the arithmetic Chow group.  He  also defined a normalized Eisenstein series $\mathcal E(\tau, s)$ of weight $3/2$,
and proved the following two formulas: Let $\widehat{\omega}$ be the (normalized) metrized Hodge bundle.
Then
  \begin{equation}\label{siegelweil1}
  2\deg \widehat\Phi_{KRY}(\tau)=\mathcal E(\tau, \frac{1}2)
   =-\frac{1}{12}+\frac{1}{8\pi \sqrt{v}}+\sum_{n=1}^{\infty} H_{0}(n)q^{n}+2\sum_{n>0}g(n, v)q^{-n^{2}},
 \end{equation}
 which is the well-known Zagier Eisenstein series (\cite{HZ}),
 and
 \begin{equation}\label{arithmeticsiegelweil1}
 4\langle \widehat\Phi_{KRY}(\tau), \widehat{\omega} \rangle_{GS}=\mathcal{E}^{\prime}(\tau, \frac{1}{2}).
 \end{equation}
Here $H_{0}(n)$ is the Hurwitz class number of binary quadratic forms of discriminant of $-n$,
$$g(n, v)=\frac{1}{16 \pi \sqrt{v}} \int_{1}^{\infty} e^{-4\pi n^{2}vr}r^{-\frac{3}{2}}dr.$$
\end{comment}

\begin{comment}
In this paper, we will extend the formulas (\ref{siegelweil1}) and (\ref{arithmeticsiegelweil1}) to  level
$4N$ with $N$ square free,  using the method in  \cite{BF1} (theta lifting and Kronecker limit formula).
\end{comment}
Let  \begin{equation}
V=\bigg\{w=\left(
  \begin{array}{cc}
    w_1 &  w_2  \\
      w_3 &  -w_1 \\
  \end{array}
\right)  \in M_{2}(\Q) : \tr(w) =0\bigg\},
\end{equation}
with quadratic form $Q(w) =N \det w =-N (w_1^2 + w_2 w_3)$, and let
   \begin{equation}
L=\bigg\{w =\left(
  \begin{array}{cc}
    b &  \frac{-a}{N}  \\
      c &  -b \\
  \end{array}
\right) \in M_{2}(\Z) \mid   a, b, c \in \Z \bigg\}
\end{equation}
 be an even integral lattice with the dual lattice $L^{\sharp}$. Then $\Spin(V)\cong \SL_2$ acts on $V$ by conjugation, and the associated Hermitian symmetric domain $\mathbb D$ is isomorphic to the upper half plane $\H$.  Since $\Gamma_0(N)$ preserves $L$ and acts on $L^\sharp/L$ trivially, we can and will identify $X_0(N)$ with the compactification of the  open orthogonal Shimura curve $\Gamma_0(N) \backslash \mathbb D$ (see Section \ref{preliminaries} for detail).

%Let \{$e_{\mu_r}, r \in \Z/2N$ be a  standard basis of $\C[L^{\sharp} / L]$  with
%$\mu_r =\diag(\frac{r}{2N}, -\frac{r}{2N})$.
 For each $\mu \in  L^\sharp/L$, denote  $L_{\mu}=\mu+L$, and
$$
L_\mu[n] =\{ w \in  L_\mu:\,  Q(w) = n\}.
$$
For $\mu \in L^\sharp/L$ and a positive rational number $n \in Q(\mu) +\Z$, let $Z(w) =\R w \in \mathbb D$ and define the devisor
\begin{eqnarray}
Z(n, \mu)&:=& \sum_{w\in  \Gamma_0(N) \backslash L_\mu[n] }Z(w) \in \CH^1(X_0(N))
\end{eqnarray}
When $\mu =\mu_r=\diag(\frac{r}{2N}, -\frac{r}{2N})$, this divisor is
the same as the Heegner divisors $P_{D,r}+ P_{D, -r} \in \hbox{CH}^1(X_0(N))$ in  \cite{GKZ}, where $D=-4Nn$ is a discriminant. For a positive real number $v>0$, let $\Xi(n, \mu, v)$ be the  Kudla Green function for $Z(n, \mu)$ in the open modular curve $Y_0(N)$ as defined in \cite{Kucentral} (see (\ref{eq:KudlaGreenFunctionDefinition}) for precise definition). The behavior of $\Xi(n, \mu, v)$ at cusps is complicated and has  not been studied before. In Sections \ref{sect:KudlaGreenFunction} and  \ref{sect:ModularCurve}, we will prove that it is smooth and of exponential decay when $D=-4Nn$ is not a square,  and has singularity along the the cusps  (Section \ref{sect:ModularCurve}) when $D\ne 0$ is a square. Even worse, when $D=0$ (which forces $\mu=0$), $\Xi(0, 0, v)$ has log-log singularity in the sense of \cite{BKK} (see Section \ref{sect:ArithIntersectionReview}). This is the most technical part of this paper.

 Let $\mathcal X_0(N)$ be the  canonical integral model over $\Z$ of $X_0(N)$ as defined in \cite{KM} (see Section  \ref{sect:ModularCurve}). In the arithmetic part of this paper, we assume $N$ is square free. For a point $x \in \mathcal X_0(N)$ over a field, since $\{ \pm 1\} \subseteq \Aut(x)$ always, we count $x$ with multiplicity $\frac{2}{|\Aut(x)|}$ for convenience.  Let $\mathcal Z(n, \mu)$ be the Zariski closure of $Z(n, \mu)$ in $\mathcal X_0(N)$, and we  obtain a family of arithmetic divisors $\widehat{\mathcal Z}(n, \mu, v)$ in $\widehat{\CH}^1_\R(\mathcal X_0(N))$---arithmetic Chow group with real coefficients in the sense of Gillet-Soul\'e as follows for $n\ne 0$:
   $$
\widehat{\mathcal Z}(n, \mu, v)
 =\begin{cases}
  (\mathcal Z(n, \mu), \Xi(n, \mu, v)) &\ff n >0,
  \\
  (0,  \Xi(n, \mu, v))  &\ff n < 0, D \ne \square,
  \\
  (g(n, \mu, v) \sum_{P \hbox{ cusps}} \mathcal P, \Xi(n, \mu, v)) &\ff  n < 0, D =\square.
 \end{cases}
$$
  Here $g(n, \mu, v)$ is some real number defined in  Theorem \ref{theo:SpecialGreen}, and $\mathcal P$ is the Zariski closure of the cusp $P$ in $\mathcal X_0(N)$. When $n=0$, the same formula (as  $D$ being a square) gives a `naive' arithmetic  Chow cycle $\widehat{\mathcal Z}(0, 0, v)^{\Naive}$, which has log-log singularity at the cusps and needs to be modified to make the `generating series' (to be defined below) modular.  Let $\widehat{\omega}_N$ be the  metrized Hodge bundle on $\mathcal X_0(N)$ with the  normalized Pettersson metric.  It has log singularity  at cusps in the sense of K\"uhn  (see Section \ref{sect:ArithIntersectionReview}). Its associated arithmetic divisor has log-log singularity at the cusps.
  It turns out  magically  that the modified arithmetic divisor
  \begin{equation}
  \widehat{\mathcal Z}(0, 0, v)=\widehat{\mathcal Z}(0, 0, v)^{\Naive} -2\widehat{\omega}_N- \sum_{p|N} \mathcal X_p^0 -(0, \log(\frac{v}{N}))
  \end{equation}
  belongs to $\widehat{\CH}^1_\R(\mathcal X_0(N))$ (Proposition \ref{prop:Chow}).  Here $\mathcal X_p^0$ (resp. $\mathcal X_p^\infty)$) is the irreducible component of $\mathcal X_0(N) \pmod p$ containing the reduction of the cusp $P_0$ (resp. $P_\infty$).  One of the main results  of this paper is the following analogue of  the  modularity theorem in \cite[Chapter 4]{KRYBook}.

  \begin{theorem} \label{theo:modularity}
 The arithmetic theta  function (for $\tau = u + iv$, and $q_\tau=e(\tau) =e^{2 \pi i \tau}$)
\begin{equation} \label{eq:GeneratingFunction1}
\widehat{\phi}(\tau) = \sum_{\mu \in L^\sharp/L} \sum_{n \in Q(\mu)+\Z} \widehat{\mathcal Z}(n, \mu, v) q_\tau^n e_\mu,
\end{equation}
is a vector valued modular form for $\Gamma'$ of weight $\frac{3}2$, valued in $\C[L^\sharp/L]\otimes \widehat{\CH}^1_\R(\mathcal X_0(N))$. Here $\Gamma'$ is the metaplectic cover of $\SL_2(\Z)$ which  acts on $\C[L^\sharp/L]$ via the Weil representation $\rho_L$ (see (\ref{eq:WeilRepresentation})) and acts on the arithmetic Chow group trivially.  Finally $\{ e_\mu:\, \mu \in  L^\sharp/L\}$ is the standard basic of $\C[L^\sharp/L]$.
  \end{theorem}

 To prove the theorem,   we will need the decomposition theorem of the arithmetic Chow group $\widehat{\CH}_\R^1(\mathcal X_0(N))$ and some  arithmetic intersection formulas as in   \cite[Chapter 4]{KRYBook}.   These intersection  formulas are important themselves, which   we now describe  briefly.

 Let
$$
E_L(\tau, s) = \sum \limits_{ \gamma ^{\prime} \in \Gamma_{\infty} ^{\prime}\diagdown
\Gamma^{\prime}} \big(v^{\frac{s-1}{2}}e_{\mu_{0}} \big)\mid_{3/2} \gamma ^{\prime}
$$
be a vector valued Eisenstein series of weight $3/2$, where  the Petersson slash operator is defined on functions $f:\H\rightarrow \C[L^{\sharp} / L]$ by
 $$\big(f \mid_{3/2} \gamma ^{\prime}\big)(\tau)=\phi(\tau)^{-3}\rho_{L}^{-1}(\gamma ^{\prime})f(\gamma \tau),$$ where $\gamma ^{\prime}=(\gamma,
 \phi)\in \Gamma^{\prime} $. Let
\begin{equation}\label{vectormodularform}
\mathcal{E}_{L}(\tau,s)=- \frac{s}{4} \pi^{-s-1}\Gamma(s)\zeta^{(N)}(2s)N^{\frac{1}{2}+\frac{3}{2}s} E_L(\tau, s)
\end{equation}
be its normalization, where
$$
\zeta^{(N)}(s) = \zeta(s) \prod_{p|N} (1-p^{-s}).
$$
\begin{remark}
In the work \cite{KRYComp} and \cite{Yazagier}, the critic point of  Eisenstein series is $s=\frac{1}{2}$. In our paper, for the convenience of computation, we define $E_L(\tau, s)$ by a shift of s.
\end{remark}
%For $p|N$, $\mathcal X_0(N) \pmod p$ is union of two irreducible components: one containing the reduction of cusp $\mathcal P_\infty$ and the other containing %the reduction of cusp $\mathcal P_0$. We denote them by $\mathcal X_p^\infty$ and $\mathcal X_p^0$ respectively.

The intersection formulas referred above are given by the following theorem.   The third formula is usually called an arithmetic Siegel-Weil formula while the first one (degree formula) is a geometric Siegel-Weil formula.

\begin{theorem} \label{maintheo} Let the notations be as above, then
\begin{align*}
\langle \widehat{\phi}(\tau), a(1) \rangle_{GS} &=\frac{1}{2} \deg(\widehat{\phi}(\tau))= \frac{1}{\varphi(N)}\mathcal E_L(\tau, 1),
\\
\langle \widehat{\phi}(\tau), \mathcal X_p^0 \rangle_{GS} &= \langle \widehat{\phi}(\tau), \mathcal X_p^\infty \rangle_{GS}=\frac{1}{\varphi(N)}\mathcal E_L(\tau, 1) \log p, \quad p|N
\end{align*}
and
$$
\langle \widehat{\phi}(\tau), \widehat{\omega}_N \rangle_{GS} =\frac{1}{\varphi(N)}\bigg(\mathcal E_L'(\tau, 1)-\sum_{p|N} \frac{p}{p-1} \mathcal E_L(\tau, 1) \log p \bigg).
$$

Here $a(1)=(0, 1) \in \widehat{\CH}_\R^1(\mathcal X_0(N))$.
\end{theorem}

\begin{comment}

\begin{remark} With this theorem and the modularity result of Gross-Kohnen-Zagier (\cite{GKZ}), one should be able to prove the modularity of $\widehat{\phi}(\tau)$ using Kudla's work  in \cite[Chapter 4]{KRYBook}. We will work on this in a sequel. Alternatively, one can prove the modularity of a similar generating function with Bruinier-Funke Green function via Borcherds product and Borcherds' modularity criteria (\cite{BoDuke}), and obtain the modularity of this generating function $\widehat{\phi}$ by the recent work of Ehlen and Sankaran \cite{ES}.
\end{remark}
\end{comment}

There are three main  ingredients in the proof of Theorem \ref{maintheo}. The first is  to analyze and  understand the behavior of Kudla's Green function $\Xi(n, \mu, v)$ for all pairs $(n, \mu) \in \Q \times L^\sharp/L$ with $Q(\mu) \equiv n \pmod 1$, in particular when $D=-4Nn \ge 0$ is a square. Here  $v>0$ is a constant. This occupies full Section \ref{sect:KudlaGreenFunction} (general case)  and the first part of Section \ref{sect:ModularCurve}. The upshot is an honest definition of the arithmetic divisors $\widehat{\mathcal Z}(n, \mu, v)^{\Naive}$ in Theorem \ref{theo:SpecialGreen}, its modification $\widehat{\mathcal Z}(n, \mu, v)$  in (\ref{eq:ArithmeticDivisor}),  and the generating function $\widehat{\phi}(\tau)$ above.

To understand $\widehat{\omega}_N$, we actually construct an explicit rational section of $\omega_N^k$, which is isomorphic to the line bundle of modular forms of weight $k$ for $k=12 \varphi(N)$ (the Euler $\varphi$-function),   i.e.,  an explicit modular form $\Delta_N$ of weight $k$  for $\Gamma_0(N)$ as follows:
\begin{equation} \label{eq:DeltaN}
\Delta_N(z) =\prod_{t|N} \Delta(t z) ^{a(t)}
\end{equation}
with
$$
a(t) = \sum_{r|t} \mu(\frac{t}r) \mu(\frac{N}r) \frac{\varphi(N)}{\varphi(\frac{N}r)},
$$
where $\mu(n)$ is the the M\"obius function. This is inspired by K\"uhn's early work on self-intersection  of $\widehat{\omega}_N$ with $N=1$ using the well-known Delta function $\Delta$.  One complication here is that $\Delta_N$ has vertical components, see Lemma \ref{divisorlemma}. This means that we will need to deal with self-intersections of vertical components (see  Section \ref{sect:Proof}).

These ingredients are enough  for the first two identities of Theorem \ref{maintheo}. To prove the last identity, we further need to compute the infinity part of the arithmetic intersection, which boils down essentially to self-intersection of $\widehat{\omega}_N$, intersection of vertical components, and  the following integral, which can  be viewed as a  theta lifting:
\begin{equation}
I(\tau,  \log\|\Delta_N\|) =\int_{X_0(N)} \log\|\Delta_N\| \Theta_L(\tau, z).
\end{equation}
Here
 $\Theta_L(\tau,z)$ is  the two variable geometric theta kernel of Kudla and Millson defined by (\ref{eq:ThetaCoefficient}), and  the Petersson norm is renormalized as
 \begin{equation} \label{eq:Petersson}
\| f(z)\| = |f(z) (4 \pi e^{-C} y)^{\frac{k}2}|= e^{-\frac{kC}2} \parallel
 f(z) \parallel_{Pet},
\end{equation} with $C=\frac{\log 4\pi +\gamma}{2}$.
 The theta   function $\Theta_L(\tau, z)$  is a vector valued modular form for $\tau$ of weight $3/2$ and modular function for the variable $z$  valued in $\Omega^{1, 1}(X_0(N))$  for $\Gamma_0(N)$.

 To connect this integral with $\mathcal E_L'(\tau, 1)$, we follow Bruinier and Funke's idea in  \cite{BF1} in two steps, given by the following two theorems, which are of independent  interest.

\begin{theorem}(Theta lifting of Eisenstein series) \label{mainresult1} Let
\begin{equation}
E(N, z, s)= \sum_{\gamma \in \Gamma_{\infty} \diagdown \Gamma_{0}(N)} (\Im {(\gamma z) })^{s},
\end{equation}
be the Eisenstein series of weight $0$ for $\Gamma_0(N)$, and let
\begin{equation}\label{normailizedeisenstein}
\mathcal{E}(N,z,s):= N^{2s}\pi^{-s}\Gamma(s)\zeta^{(N)}(2s)E(N, z, s)
\end{equation}
be its normalization. Then
 \begin{eqnarray}
 {I}(\tau, \mathcal{E}(N,z,s))=I(\tau, \mathcal E(N, w_N z, s))= \zeta^{\ast}(s)\mathcal{E}_{L}(\tau,s),\nonumber
 \end{eqnarray}
 where $w_N =\kzxz {0} {-1} {N} {0}$ and $\zeta^{\ast}(s)= \pi^{-\frac{s}{2}}\Gamma(\frac{s}{2})\zeta(s).$
\end{theorem}

\begin{theorem} (Kronecker Limit formula for $\Gamma_0(N)$) \label{firstlimitformula}   Let
the notations be as above, then
  one has
$$
\lim _{s\rightarrow 1} \bigg( \mathcal{E}(N, z, s)-\varphi(N)\zeta^{\ast}(2s-1)\bigg) \nonumber\\
%&=&-\varphi(N)\log \big(\sqrt {y}\mid\eta_{\Gamma_
%{0}(N)}\mid^{2}\big)\nonumber\\
=-\frac{1}{12}\log \big(y^{6\varphi(N)}\mid \Delta_{N}(z)\mid\big), \nonumber
$$
and
$$
\lim _{s\rightarrow 1} \bigg( \mathcal{E}(N, w_N z, s)-\varphi(N)\zeta^{\ast}(2s-1)\bigg) \nonumber\\
%&=&-\varphi(N)\log \big(\sqrt {y}\mid\eta_{\Gamma_
%{0}(N)}\mid^{2}\big)\nonumber\\
=-\frac{1}{12}\log \big(y^{6\varphi(N)}\mid \Delta_{N}^0(z)\mid\big),
$$
where $\Delta_N^0 =\Delta_N|w_N$.
\end{theorem}

Combining the previous theorems, we   obtain

\begin{theorem}\label{derivatives} One has
$$
I(\tau, 1) = \frac{2}{\varphi(N)} \mathcal E_L(\tau, 1)
$$
and
$$
 {I}(\tau, \log \parallel \Delta_{N} \parallel)=  I(\tau, \log\| \Delta_N^0\|)= - 12 \mathcal{E}_{L}^{\prime}(\tau,1).
 $$
\end{theorem}

This paper is organized in two parts as follows. In Part I,  we prove Theorem \ref{mainresult1} after setting up notation and introduce the  theta lifting (\ref{eq:ThetaLifting}) in Section \ref{sect:ThetaLifting}. In Section \ref{sect:Kronecker}, we study some basic properties of $\Delta_N$ and prove the Kronecker limit formula Theorem \ref{firstlimitformula} and then Theorem \ref{derivatives}. We also prove some properties of $\Delta_N$ needed in Part II.

In  Part II, we first review arithmetic divisors with log-log singularity, metrized line bundles with log singularity, and arithmetic intersection  in Section \ref{sect:ArithIntersectionReview} following K\"uhn \cite{Kuhn2} and   Burgos Gil,  Kramer and  U.  K\"uhn \cite{BKK}. In Section \ref{sect:KudlaGreenFunction}, we study the behaviors of Kudla's Green functions at cusps in a more general setting (see Theorem \ref{theo:Singularity}). In Section \ref{sect:ModularCurve}, we focus on the modular curve $\mathcal X_0(N)$ for square free $N$, and prove Theorem \ref{theo:SpecialGreen}. We also decompose Theorem \ref{maintheo} into two results, which we prove in Section \ref{sect:ModularCurve}, and a `horizontal intersection' theorem  Theorem \ref{theo:horizontal}, which we will prove in Section \ref{sect:Proof}. In last section, we will prove the modularity theorem (Theorem \ref{theo:modularity}).

Finally, we remark that the technical condition $N$ being square free is only needed in the arithmetic part, mainly  to avoid the complication of special fiber of $\mathcal X_0(N)$ at  $p^2|N$ when $N$ is not square free. A different way to prove the modularity theorem is to first prove a modularity theorem for a similar generating function with Bruinier's Green functions  and then to show the generating function of the difference of the two Green functions is modular as done in recent work of Ehlen and Sankaran \cite{ES}.

\textbf{Acknowledgments.} Add later.

\part{Theta lifting  and Kronecker limit formula}

\section{Basic set-up and theta lifting } \label{preliminaries} \label{sect:ThetaLifting}

 Let  \begin{equation}
V=\bigg\{w=\left(
  \begin{array}{cc}
    w_1 &  w_2  \\
      w_3 &  -w_1 \\
  \end{array}
\right)  \in M_{2}(\Q) : \tr(w) =0\bigg\},
\end{equation}
with the quadratic form $Q(w)=N \det{w}=-Nw_2w_3-Nw_1^{2}$, which  has signature $(1, 2)$. Let $L$ be the even integral lattice defined in the introduction with dual lattice $L^\sharp$. We will identify
$$
\Z/2N\Z \cong L^\sharp/L,  \quad r\mapsto \mu_r = \kzxz {\frac{r}{2N}} {0} {0} {-\frac{r}{2N}}.
$$
 Let $G= \SL_2\cong \Spin(V)$ act on $V$  by conjugation, i.e.,
 $g\cdot w=gwg^{-1}$. Notice that $\Gamma_0(N)$ preserves $L$ and acts on $L^\sharp/L$ trivially.
Let  $\D$ be the Hermitian domain of  positive real lines in $V_\R$:
$$
\D=\{z \subset V_{\R}; \dim z=1 ~and~(~,~) \mid_{z} >0\}.
$$
The following lemma can be easily checked and is left to the reader.

\begin{lemma}  \label{lem2.1} For $z=x+ iy$, define
$$w(z) = \frac{1}{\sqrt{N} y} \left(
  \begin{array}{cc}
  -x  & z\overline{z}\\
     -1&x\\
  \end{array}
\right).
$$
Then $z \mapsto [w(z)]=\R w(z)$ gives an $G(\R)$-invariant isomorphism between the upper half plane $\H$ and $\mathbb D$. It induces thus an
isomorphism between
$Y_0(N) =\Gamma_0(N) \backslash \H$ and $\Gamma_0(N) \backslash \mathbb D$.
\end{lemma}

Let $X_0(N)$ be the usual compactification of $Y_0(N)$.
Let $\Mp_{2, \R}$ be the metaplectic double cover of $\SL_2(\R)$, which can be  realized as pairs $(g, \phi(g, \tau))$, where $g=\abcd \in \SL_2(\R)$, $\phi(g,
\tau)$ is a holomorphic function of $\tau \in \H$ such that $\phi(g, \tau)^2 = j(g, \tau) = c\tau +d$. Let $\Gamma'$ be the preimage of
$\Gamma=\SL_2(\Z)$ in $\Mp_{2, \R}$, then $\Gamma'$ is generated by
$$
S=\left( \kzxz {0} {-1} {1} {0}, \sqrt \tau \right)  \quad  T= \left( \kzxz {1} {1} {0} {1} , 1 \right).
$$
We denote the standard basis of $S_L=\C[L^\sharp/L]$ by $\{ e_\mu=L_\mu:\, \mu \in L^\sharp/L\}$.  Then there is a Weil representation $\rho_L$ of $\Gamma'$ on $\C[L^\sharp/L]$ given by  (
\cite{Bo})
\begin{eqnarray} \label{eq:WeilRepresentation}
&\rho_{L} (T)e_{\mu}&=e(Q(\mu))e_{\mu} ,\\
&\rho_{L}(S) e_{\mu}&= \frac{e(\frac{1}8)}{\sqrt{\vert L^{\sharp} / L \vert}}\sum \limits_{\mu^{\prime} \in L^{\sharp} / L }e(-(\mu,
\mu^{\prime}))e_{\mu^{\prime}}\nonumber .
\end{eqnarray}
This Weil representation $\rho_L$ is naturally connected to the Weil representation
$\omega$ of $\Mp_{2, \A}$ on $S(V_\A)$, see \cite{BHY} for explanation.
\begin{comment}
Now we are ready to review the theta kernel $\Theta_L(\tau, z)$ and theta lifting in this case for $\tau, z\in \H$. We emphasize the different roles which $\tau$ and $z$ play. $\tau$ is in the modular curve $\SL_2(\Z) \backslash \H$ associated to $\SL_2$ while $z \in \Gamma_0(N) \backslash \H$ is associate to the orthogonal Shimura curve $\Spin(V)$. We follow Kudla and Millson in  \cite{KMi} (see also  \cite[section 3]{BF1} in this special case).
\end{comment}

Following Kudla and Millson  (\cite{KMi}, \cite[Section 3]{BF1}),
 we  decompose for  $z=x+iy \in \H$,
$$
V_\R =\R w(z) \oplus w(z)^\perp, \quad w = w_z + w_{z^\perp},
$$
and define
$R(w, z) =-(w_{z^\perp}, w_{z^\perp})$, and the majorant
$$
(w, w)_z= (w_z, w_z) + R(w, z).
$$
Since $Q(w(z)) =1$, it is easy to check
\begin{align} \label{eq3.5}
R(w, z)&=\frac{1}{2}(w, w(z))^{2}-(w, w),
\\
(w, w)_{z}&=(w, w(z))^{2}- (w, w). \notag
\end{align}
For  $w =\left(
  \begin{array}{cc}
    w_{1}&w_{2}\\
     w_{3}&-w_{1}\\
  \end{array}
\right) \in V_{\R},$  we have
\begin{equation} \label{formula1}
(w, w(z))=-\frac{\sqrt{N}}{y}(w_3z\overline{z}-w_1(z+\overline{z})-w_2).
\end{equation}

\begin{comment}
Let $\mu(z) = \frac{dx dy}{y^2}$ be the Poincare measure on $\H$.  For $\tau= u + iv \in \H$, let
$ g_{\tau}^{'} =(n(u) m(\sqrt v), 1) \in \Mp_{2, \R}$, where
$$
n(u) =\kzxz {1} {u} {0} {1}, \quad m(a) =\kzxz {a} {0} {0} {a^{-1}}.
$$
\end{comment}
Let  $\mu(z) =\frac{dx \, dy}{y^2}$,
\begin{align} \label{eq3.5}
\varphi^{0}(w, z)&=\bigg((w, w(z))^{2}- \frac{1}{2\pi}\bigg)e^{-2\pi R(w, z)}\mu(z), \notag
\\
\varphi(w, \tau, z) &= e(Q(w)\tau) \varphi^0(\sqrt v w, z) ,
\end{align}
which is a Schwartz function on $V_{\R}$ valued in $\Omega^{1,1}(\D)$ constructed by Kudla and Millson in \cite{KMi}.
Finally, let
\begin{equation} \label{eq:ThetaCoefficient}
\Theta_L(\tau, z) =   \sum_{\mu \in  L^\sharp/L}  \theta_\mu(\tau, z) e_\mu
\end{equation}
be the vector valued Kudla-Millson theta function, where
\begin{align}
\theta_\mu(\tau, z) &=\sum_{n, \mu} \sum_{ w \in L_\mu[n]} \varphi( w, \tau, z)   \label{eq:ScalarTheta}
\\
 &= \sum_{n\in \Q, Q(\mu) \equiv n \pmod 1} \omega(n, \mu, v)(z) q^n
  + \begin{cases} 0  &\ff \mu \ne 0,
  \\
  -\frac{1}{2\pi} \mu(z)   &\ff \mu = 0,
  \end{cases}  \notag
\end{align}
with ($q=q_\tau=e(\tau)$)
\begin{equation} \label{eq:Differential}
\omega(n,\mu,  v)(z) =\sum_{0\ne  w \in L_\mu[n]}\varphi^{0} (v^{\frac{1}{2}}w, z) \in  \Omega^{1, 1}(X_\Gamma).
\end{equation}
It is known that $\Theta_L(\tau, z)$ is a nonholomorphic modular form of weight $3/2$  of $(\Gamma',  \rho_L)$ valued  in $ \Omega^{1, 1}(X_{\Gamma}) \otimes \C[L^\sharp/L] $ as a function of $\tau$. It is $\Gamma_0(N)$-invariant as a function of $z$.

The following result of Funke about behavior of $\theta_\mu$ as $z$ goes to the boundary (cusp) is important to our definition of theta lifting.

\begin{proposition} \cite [ Proposition 4.1]{BF1} \label{convergence}
Fix  $\mu \in  L^{\sharp} / L$ and $\tau \in \H$. Let  $l=\sigma_l(\infty)$ be a cusp of $X_0(N)$.  As a function of $z=x+ i y \in \H =\mathbb D$,
the theta function (recall $z= x + iy $)
  $$
  \theta_{\mu}( \tau, \sigma_{l}z)=O( e^{-Cy^{2}}), ~as~y\longrightarrow \infty$$
holds uniformly in $x$ for some constant $ C> 0$.

\end{proposition}

For a   (non-holomorphic) modular function  $f(z)$  for $\Gamma_0(N)$ (viewed as subgroup of the Spin group) with moderate growth, the theta
lifting
\begin{equation} \label{eq:ThetaLifting}
I(\tau, f) =\int_{\Gamma_0(N) \backslash \mathbb D} f(z) \Theta_L(\tau, z)  =\sum_{\mu \in L^\sharp/L} I_\mu(\tau, f) e_\mu
\end{equation}
is absolutely convergent by Proposition \ref{convergence} and is a (non-holomorphic) weight $3/2$  modular form of $\Gamma'$ with values in
$\C[L^\sharp/L]$.

{\bf Proof of Theorem \label{eq:ScalarTheta}\ref{mainresult1}}:
%\begin{proof}
Firstly,  we compute the theta series:
\begin{eqnarray}
&&\theta_{\mu_r}( \tau, z )= \sum_{w \in L_{\mu_r}} \varphi(w, \tau, z)\nonumber \\
&=& \sum\limits_{ w_{1} \in \Z+\frac{r}{2N},  n, w_3 \in \Z }
 \bigg( \frac{v}{Ny^{2}}\big(N(w_3z\overline{z}-w_1(z+\overline{z}))-n\big)^{2}-\frac{1}{2\pi}\bigg) \nonumber \\
&& \times e(-N\overline{\tau}w_1^{2})e(-\overline{\tau}w_3n)e\bigg(\frac{iv}{2Ny^{2}} \big( N(w_3 z \overline{z}-
w_1(z+\overline{z}))-n\big)^{2}\bigg)\mu(z).\nonumber
\end{eqnarray}
Let
$$f(X)=  \big( \frac{vX^{2}}{Ny^{2}}-\frac{1}{2\pi}\big)
e(-\overline{\tau}w_3X)e\big(\frac{ivX^{2}}{2Ny^{2}}\big),$$
then
\begin{eqnarray}
\widehat{f}(m)&=&\int_{-\infty}^{\infty}f(X)e(-mX) dX
= -\frac{N^{\frac{3}{2}}y^{3}}{v^{\frac{3}{2}}}(\overline{\tau}w_3+m)^{2}e\big( \frac{iNy^{2}}{2v}(\overline{\tau}w_3+m)^{2}\big).\nonumber
\end{eqnarray}
Write  $t=N(w_3 z \overline{z}- w_1(z+\overline{z}))$.  Applying
 the Poisson summation formula, we obtain
\begin{eqnarray}
\theta_{\mu_{r}}( \tau, z )&=& \sum \limits_{ w_{1} \in \Z+\frac{r}{2N},  m, w_3 \in \Z }
e(-N\overline{\tau}w_1^{2}) e(-\overline{\tau}w_3 t) \widehat{f}(m)e(-mt)\mu(z) \nonumber\\
&=&-\frac{N^{\frac{3}{2}}y^{3}}{v^{\frac{3}{2}}}\sum \limits_{ w_{1} \in \Z+\frac{r}{2N},  m, w_3 \in \Z } (\overline{\tau}w_3+m)^{2}
e\big(-N\overline{\tau}(w_1-w_3x)^{2}\big)\nonumber \\
&\times & e\big(2N(w_1-w_3m/2)mx\big)exp\big(-\frac{\pi Ny^{2}}{v}\mid m+w_3\tau\mid^{2}\big) \mu(z).\nonumber
\end{eqnarray}

As in \cite[Section 4]{Bo},  we define for $\alpha, \beta \in \Q$
\begin{equation}
\Theta_L(\tau, \alpha, \beta)= \sum_{r \in \Z/2N} \sum_{w_1 \in \frac{r}{2N}+\Z }e(-\overline{\tau}(w_1+ \beta)^{2})
e(-\alpha(2w_1+\beta))e_{\mu_r}.
\end{equation}

For $\gamma^{\prime}=\bigg(\left(
  \begin{array}{cc}
    a& b\\
     c & d\\
  \end{array}
\right), \sqrt{c\tau +d}\bigg) \in \Gamma^{\prime} ,$  it is easy to check
\begin{equation} \label{eq2.11}
\Theta_L(\tau, ndx, -ncx)=(c\overline{\tau}+d)^{-\frac{1}{2}}\rho_{L}^{-1}(\gamma^{\prime})\Theta_L(\gamma^{\prime}\tau, nx, 0).
\end{equation}

We continue the calculation:

\begin{eqnarray}
&&\Theta_L (\tau, z)\nonumber
 \\
  &=&-\frac{N^{\frac{3}{2}}y^{3}}{v^{\frac{3}{2}}}\sum \limits_{  m, w_3 \in \Z
}(\overline{\tau}w_3+m)^{2}e^{\big(-\frac{\pi Ny^{2}}{v}\mid m+w_3\tau\mid^{2}\big)}\Theta_L(\tau, mx, -w_3x)\mu(z) \nonumber
\\
&=&-\frac{N^{\frac{3}{2}}y^{3}}{v^{\frac{3}{2}}} \sum_{n=1}^{\infty}n^{2}\sum \limits_{ c, d \in \Z, (c, d)=1
}(c\overline{\tau}+d)^{2}e^{\big(-\frac{\pi Ny^{2}n^{2}}{v}\mid c\overline{\tau}+d\mid^{2}\big)}\Theta_L(\tau, ndx, -ncx)\mu(z) \nonumber
\\
&=&-\frac{N^{\frac{3}{2}}y^{3}}{v^{\frac{3}{2}}} \sum_{n=1}^{\infty}n^{2}\sum \limits_{ \gamma^{\prime} \in \Gamma_{\infty} ^{\prime}\diagdown
\Gamma^{\prime}}(c\overline{\tau}+d)^{\frac{3}{2}}e^{\big(-\frac{\pi Ny^{2}n^{2}}{v}\mid
c\overline{\tau}+d\mid^{2}\big)}\rho_{L}^{-1}(\gamma^{\prime})\Theta_L(\gamma^{\prime}\tau, nx, 0)\mu(z). \nonumber
\end{eqnarray}
Unfolding the integral,  we have for $\Re(s)>1$
\begin{eqnarray}\label{thetalift}
%\\
&&\emph{I}(\tau, E(N, z, s))=\int_{\Gamma_{\infty}\diagdown\H} \Theta_L (\tau, z)y^{s}\nonumber\\
&=&-v^{-\frac{3}{2}}N^{\frac{3}{2}}\sum_{n=1}^{\infty}n^{2}\sum \limits_{ \gamma^{\prime} \in \Gamma_{\infty} ^{\prime}\diagdown
\Gamma^{\prime}}(c\overline{\tau}+d)^{3/2}\int_{0}^{\infty}e^{\big(-\frac{\pi Ny^{2}n^{2}}{v}\mid c\overline{\tau}+d\mid^{2}\big)}y^{s+1}dy\nonumber
\\
&&\times \rho_{L}^{-1}(\gamma^{\prime})\int_{0}^{1}\Theta_L(\gamma^{\prime}\tau, nx, 0)dx.\nonumber
\end{eqnarray}
It is easy to check that
$$
\int_{0}^{1}\Theta_L(\gamma^{\prime}\tau, nx, 0)dx = e_{\mu_0}.
$$
So

%By (\ref{vectormodularform}), it equals to ({\bf If $\mu_{r}$ is nonzero, then the theta integral is zero; when $\mu_{r}=0$, the theta integral is $1$.})
\begin{eqnarray}
&&\int_{\Gamma_{\infty}\diagdown\H} \Theta_L (\tau, z)y^{s}\nonumber
 \\
&=&-\frac{1}{2}v^{-\frac{3}{2}}N^{\frac{3}{2}}\sum_{n=1}^{\infty}n^{2}\sum \limits_{ \gamma^{\prime} \in \Gamma_{\infty} ^{\prime}\diagdown \Gamma^{\prime}}
\frac{v^{\frac{s+2}{2}}(c\overline{\tau}+d)^{3/2}\Gamma\big(\frac{s}{2}+1\big)}{\pi^{\frac{s+2}{2}}\mid
c\tau+d\mid^{s+2}N^{\frac{s+2}{2}}n^{s+2}}\rho_{L}^{-1}(\gamma^{\prime})e_{\mu_{0}} \nonumber\\
&=&-\frac{1}{2}N^{\frac{1-s}{2}} \zeta(s)\Gamma\big(\frac{s}{2}+1\big)\sum \limits_{ \gamma ^{\prime} \in \Gamma_{\infty} ^{\prime}\diagdown
  \Gamma^{\prime}} \frac{v^{\frac{s-1}{2}}(c\overline{\tau}+d)^{3/2}}{\pi^{\frac{s}{2}+1}\mid
  c\tau+d\mid^{s+2}}\rho_{L}^{-1}(\gamma^{\prime})e_{\mu_{0}} \nonumber\\
 &=&-N^{\frac{1-s}{2}} \frac{s}{4\pi}\zeta^{\ast}(s)\sum \limits_{ \gamma ^{\prime} \in \Gamma_{\infty} ^{\prime}\diagdown \Gamma^{\prime}}
 \big(v^{\frac{s-1}{2}}e_{\mu_{0}} \big)\mid_{3/2,L} \gamma ^{\prime}. \nonumber
\end{eqnarray}
In summary, we have proved
\begin{equation}
\emph{I}(\tau, E(N, z, s))=-N^{\frac{1-s}{2}} \frac{s}{4\pi}\zeta^{\ast}(s)E_L(\tau, s).\nonumber
\end{equation}
 or equivalently,
 \begin{eqnarray}\label{equliftone}
 {I}(\tau, \mathcal{E}(N,z,s))=\zeta^{\ast}(s)\mathcal{E}_{L}(\tau,s).
 \end{eqnarray}
%\end{proof}
It is easy to check by definition that
$$
\theta_L(\tau, z) = \theta_L(\tau, w_N(z)).
$$
This implies that
$$
I(\tau, \mathcal E(N, w_N(z), s)) =I(\tau, \mathcal E(N, z, s)).
$$
This proves the theorem.

Taking residue  of both sides of the equation (\ref{equliftone}) at $ s=1$, we have the following result.
\begin{corollary}\label{cor:Value}
 \begin{equation}{I}(\tau, 1)=\frac{2}{\varphi(N)}\mathcal{E}_{L}(\tau,1).
 \end{equation}
\end{corollary}

\section{Kronecker limit formula for the group $\Gamma_{0}(N)$ }\label{Firstlimit} \label{sect:Kronecker}

We need some preparation before proving Theorem \ref{firstlimitformula}---the Kronecker Limit formula for $\Gamma_0(N)$. These auxiliary results will also be used in Section \ref{sect:ModularCurve} and should be of independent interest.

Let
\begin{equation}
C_{N}(n)= \sum_{a=1, (a, N)=1}^{N}e(\frac{an}{N})
\end{equation}
be the Ramanujan sum.  It has the following properties according to Kluver (\cite[p.411]{Kl}).

\begin{lemma}  (Kluver)\label{lem:RamanujanSum} Let $t=(N, n)$ be the greatest common divisor of $N$ and $n$. Then one has
\begin{align*}
C_N(n) &= \frac{\varphi(N)}{\varphi(\frac{N}t)} C_{\frac{N}t}(1)
\\
C_N(n) &=\sum_{r|t} \mu(\frac{N}r) r.
\end{align*}
Here $\varphi$ is the classical Euler $\varphi$-function, and  $\mu(t)$ is the well-known M\"obius function.
In particular, one has $C_N(1) =\mu(N)$.
\end{lemma}
\begin{comment}
\begin{proof} The first identity is almost obvious from the definition. Indeed, write $n=tn_1$ and $N=t N_1$ with $(n_1, N_1)=1$
\begin{align*}
C_N(n) &=\sum_{a \in (\Z/N)^\times} e(\frac{an_1}{N_1})
  \\
   &=\sum_{a \in (\Z/N_1)^\times} \sum_{b \in (1+N_1\Z)/(1+N\Z)} e(\frac{a}{N_1})
   \\
   &=\frac{\varphi(N)}{\varphi(N_1)} C_{N_1}(1),
\end{align*}
as claimed. The second identity is a classical formula of Kluver (\cite[p.411]{Kl}).
\end{proof}
\end{comment}
%The following elementary lemma will  be used in the next proposition and might be of independent interest.

\begin{lemma} \label{lemAt} For a positive integer $N$ and a divisor $t$ of $N$, let
$$
a_N(t)=\sum_{r|t} \mu(\frac{t}r) \mu(\frac{N}r) \frac{\varphi(N)}{\varphi(\frac{N}r)}
$$
be as in the introduction. Then the following are true.
\begin{enumerate}

\item If $Q\| N$, i.e, $Q|N$ and $(Q, N/Q)=1$, write $t=t_1 t_2$. Then  $a_N(t) =a_Q(t_1) a_{N/Q}(t_2)$.

\item One has
\begin{align*}
\sum_{t|N} a_N(t) &= \varphi(N),
\\
\sum_{t|N} t a_N(t) &= N \varphi(N) \prod_{p|N} (1+p^{-1}),
\\
\sum_{t|N} t^{-1} a_N(t) &= 0 \quad \hbox{when } N >1.
\end{align*}
\end{enumerate}
\end{lemma}
\begin{proof} (1) is clear. For (2), we check the second identity and leave the others to the reader.  We drop the subscript $N$ from now on as $N$ will be fixed.  One has
\begin{comment} It is easy to check
 $$
\sum_{t|N} a(t) =  \sum_{r|N} \frac{\varphi(N)}{\varphi(\frac{N}r)} \mu(\frac{N}r) \sum_{r|t|N} \mu(\frac{t}r) =  \varphi(N).
$$
Next,
\end{comment}
\begin{align*}
\sum_{t|N} t a(t)
 &= \sum_{t|N} t \sum_{r|t} \mu(\frac{t}r) \mu(\frac{N}r) \frac{\varphi(N)}{\varphi(\frac{N}r)}
 \\
 &=\varphi(N) \sum_{r|N} \frac{\mu(\frac{N}r)}{\varphi(\frac{N}r)} \sum_{t|\frac{N}r} rt \mu(t)     \quad (\hbox{ replacing $t$ by $rt$})
 \\
 &=N \varphi(N) \sum_{\substack{ r|N \\ \hbox{r square free}}} \frac{\mu(r)}{r \varphi(r)} \sum_{t|r} t \mu(t)     \quad (\hbox{ replacing $N/r$ by
 $r$})
 \\
 &= N \varphi(N) \sum_{\substack{ r|N \\ \hbox{r square free}}} \frac{1}r
=N \varphi(N) \prod_{p|N} (1 +p^{-1}).
\end{align*}
\begin{comment}
Finally,
\begin{align*}
\sum_{t|N} t^{-1} a(t)
&=\varphi(N) \sum_{r|N} \mu(\frac{N}r) \varphi(\frac{N}r)^{-1} \sum_{r|t|N} t^{-1} \mu(\frac{t}r)
\\
&=\varphi(N)  \sum_{r|N} r^{-1} \mu(\frac{N}r) \varphi(\frac{N}r)^{-1}  \sum_{t|\frac{N}r} \frac{\mu(t)}{t}
\\
&=\varphi(N) \sum_{r|N} r^{-1} \mu(\frac{N}r) \varphi(\frac{N}r)^{-1} \prod_{t|\frac{N}r} (1-p^{-1})
\\
&=\frac{\varphi(N)}{N} \sum_{r|N} \mu(\frac{N}r)
\\
&=0
\end{align*}
when  $N>1$.
\end{comment}

\end{proof}

\begin{proposition} \label{prop: DeltaN} Let $\Delta_N(z)$ be defined as in (\ref{eq:DeltaN}). Then ($q_z=e(z)$)
$$
\Delta_N(z)= q_z^{N\varphi(N)\prod_{p \mid N}(1+p^{-1})}\prod_{n\geq 1}(1-q_z^{n})^{24C_{N}(n)}.
$$
\end{proposition}
\begin{comment}
be the generalized Delta function of level $N$. Then
$$
\Delta_N(z) =\prod_{t|N} \Delta(t z) ^{a(t)}.
$$
In particular,
the generalized Delta function $\Delta_N(z)$
is a modular form in $M_{12\varphi(N)}(\Gamma_{0}(N))$ with integral Fourier coefficients.   Moreover,  $\Delta_N(z)$ vanishes at the cusp $\infty$
and does not vanish  at other cusps.
\end{comment}
\begin{proof}
Let
\begin{align*}
\tilde\Delta_N(z) &= \prod_{n} (1-q_z^n)^{C_N(n)},
\hbox{ and } \,  \tilde\Delta(z) = \prod_{n=1}^\infty (1-q_z^n).
\end{align*}
Suppose  that there are numbers $b(t)$ with
$$
\tilde\Delta_N(z) =\prod_{t|N} \tilde\Delta(tz)^{b(t)},
$$
which implies by Lemma \ref{lem:RamanujanSum}
\begin{align*}
\prod_{t|N} \prod_{(n, \frac{N}t) =1} (1-q_z^{tn})^{\frac{\varphi(N)}{\varphi(N/t)} \mu(N/t)}
&= \prod_{t|N} \prod_{n} (1-q_z^{tn})^{b(t)}
\\
&=\prod_{t|N} \prod_{t'|\frac{N}t}  \prod_{ (n, \frac{N}{tt'})=1} (1-q_z^{tt' n})^{b(t)}
\\
&=\prod_{r|N} \prod_{t|r} \prod_{ (n, \frac{N}r)=1} (1-q_z^{r n})^{b(t)}
\\
&=\prod_{r|N}\prod_{ (n, \frac{N}r)=1} (1-q_z^{r n})^{\sum_{t|r} b(t)}.
\end{align*}
So for every $r|N$, one has
\begin{equation}\label{Deltapower}
\sum_{t|r} b(t) = \frac{\varphi(N)}{\varphi(N/r)} \mu(N/r).
\end{equation}
By M\"obius inverse formula, one has
$$
b(t) = \sum_{r|t} \mu(\frac{t}r)\mu(\frac{N}r) \frac{\varphi(N)}{\varphi(\frac{N}r)}=a(t).
$$
So we have proved that
$$
\tilde\Delta_N(z) =\prod_{t|N} \tilde\Delta(tz)^{a(t)}.
$$
Combining this with Lemma \ref{lemAt} (2), we obtain the lemma.
\end{proof}

Recall (\cite{Miy}) that cusps of  $X_0(N)$ are given by $P_{\frac{aQ}N} =\frac{aQ}N$, where $Q|N$ and $a \in  (\Z/ (Q, N/Q)\Z)^\times$.  In particular, when $Q\| N$,  i.e., $Q|N$ and $(Q, N/Q)=1$, there is a unique cusp $P_{\frac{Q}N}$ associated to it. $Q=1$ is associated to $P_\infty=P_{\frac{1}N}$, and $Q=N$ is associated to $P_0=P_1$. Assume $Q \| N$, and let
$$
W_Q= \kzxz {\alpha} {\beta} {\gamma \frac{N}Q} {Q \delta} \kzxz {Q} {0} {0} {1}, \quad  \kzxz {\alpha} {\beta} {\gamma \frac{N}Q} {Q \delta} \in \Gamma_0(N/Q)
$$
be an Atkin-Lehner involution matrix with $w_Q(P_\infty) =P_{\frac{Q}N}$. Notice that when $N$ is a square free, $P_{\frac{Q}N}$, $Q|N$,  give all the cusps of $X_0(N)$.  The following proposition gives Fourier expansion of $\Delta_N$ at cusps associated  to $Q\|N$.

\begin{proposition} \label{prop:Atkin-Lehner} Assume $Q\| N$. For $t|N$, write $t_0=(t, Q)$ for their greatest common divisor. Then
$$
\Delta_N|W_Q(z) = C_Q \prod_{t|N} \Delta(\frac{t}{t_0} \frac{Q}{t_0} z)^{a_N(t)}
$$
where
$$
C_Q= Q^{6 \varphi(N)} \prod_{t_0|Q} t_0^{-12 \varphi(\frac{N}Q) a_Q(t_0)}.
$$
In particular, $\ord_p C_Q =0$ for $p\nmid Q$.
Moreover, $\Delta_N(z)$ does not vanish at cusps associated to $Q\| N$ (with $Q\ne 1$).
\end{proposition}
\begin{proof} Write $k=12 \varphi(N)$, $t=t_0 t_1$, and $Q=t_0 Q_1$. Then
\begin{align*}
\Delta_N|W_Q(z) &= \frac{Q^{\frac{k}2}}{(\gamma N z +Q \delta)^k} \prod_{t|N} \Delta(\frac{\alpha Q t z+ t \beta}{\gamma N z + Q \delta})^{a_N(t)}
\\
 &=\frac{Q^{\frac{k}2}}{(\gamma N z +Q \delta)^k} \prod_{t|N} \Delta(\frac{ \alpha t_0 (t_1 Q_1z) +t_1 \beta}{\gamma \frac{N}{t_1 Q} (t_1Q_1 z) + Q_1 \delta})^{a_N(t)}
 \\
 &= A_Q \prod_{t|N} \Delta(t_1 Q_1z)^{a_N(t)},
\end{align*}
where (recall Lemma \ref{lemAt})
$$
A_Q = Q^{\frac{k}2} \prod_{t|Q} t_0^{-12a_N(t)}=C_Q.
$$
 On the other hand, the leading $q$-power exponent of $\Delta_N|W_Q$ is given by the above calculation  (recall again Lemma \ref{lemAt})
\begin{align*}
\sum_{t|Q} t_1 Q_1 a_N(t)
 &= \sum_{t_0|Q} \frac{Q}{t_0} a_Q(t_0) \sum_{t_1|\frac{N}Q} t_1 a_{\frac{N}Q}(t_1)
 \\
 &=\begin{cases}
    0 &\ff Q >1,
    \\
    N \varphi(N) \prod_{p|N} (1+p^{-1}) &\ff Q=1.
    \end{cases}
\end{align*}
This proves the result.
\end{proof}

\begin{comment}

\begin{corollary} \label{cor:DeltaN0} Let
\begin{equation}
\Delta_N^0(z) =\Delta_N(z)|w_N = C_N \prod_{t|N} \Delta(tz)^{a(\frac{N}t)} \in M_k(N).
\end{equation}
Then it has vanishing order $\varphi(N) N \prod_{p|N} (1+p^{-1})$ at the cusp $P_0$ and does not vanish at other cusps.
\end{corollary}

Now we are ready to  prove the  Kronecker  limit formula for $\Gamma_0(N)$  in  Theorem \ref{firstlimitformula}, which we restate here for convenience.

\end{comment}

\begin{comment}
\begin{theorem} (Kronecker Limit formula) \label{limitformula}   Let $\mathcal E(N, z, s)$ be the Eisenstein series defined in Theorem \ref{mainresult1}.  Then one has
$$
\lim _{s\rightarrow 1} \bigg( \mathcal{E}(N, z, s)-\varphi(N)\zeta^{\ast}(2s-1)\bigg) \nonumber\\
%&=&-\varphi(N)\log \big(\sqrt {y}\mid\eta_{\Gamma_
%{0}(N)}\mid^{2}\big)\nonumber\\
=-\frac{1}{12}\log \big(y^{6\varphi(N)}\mid \Delta_{N}(z)\mid\big), \nonumber
$$
and
$$
\lim _{s\rightarrow 1} \bigg( \mathcal{E}(N, w_N z, s)-\varphi(N)\zeta^{\ast}(2s-1)\bigg) \nonumber\\
%&=&-\varphi(N)\log \big(\sqrt {y}\mid\eta_{\Gamma_
%{0}(N)}\mid^{2}\big)\nonumber\\
=-\frac{1}{12}\log \big(y^{6\varphi(N)}\mid \Delta_{N}^0(z)\mid\big).
$$
\end{theorem}
\end{comment}
{\bf Proof of Theorem \ref{firstlimitformula}}:
%\begin{proof}
Recall the Whittaker function (\cite[Chapter 2]{WD}) for $ y> 0$ and $\alpha, \beta \in \C $:
\begin{equation}
W(y, \alpha, \beta) =\Gamma (\beta)^{-1} \int _{0}^{\infty} (1+h)^{\alpha -1} h^{\beta -1} e^{-yh} dh.
\end{equation}
Define
\begin{eqnarray}\label{modfiedwhittaker}
&& t_{n}(y, \alpha, \beta)\nonumber \\= && \begin{cases}
   i^{\beta - \alpha} (2\pi)^{\alpha + \beta} n^{\alpha + \beta -1} e^{-2 \pi ny} \Gamma (\alpha)^{-1} W(4 \pi ny, \alpha, \beta), &\ff  n > 0, \\
i^{\beta - \alpha} (2\pi)^{\alpha + \beta}\mid n\mid^{\alpha + \beta -1} e^{-2 \pi\mid n \mid y} \Gamma (\beta)^{-1} W(4 \pi\mid n\mid y, \beta,
\alpha),  &\ff  n<0,\\
 i^{\beta - \alpha} (2\pi)^{\alpha + \beta} \Gamma (\alpha)^{-1}  \Gamma (\beta)^{-1} \Gamma(\alpha+\beta-1)(4\pi y)^{1-\alpha-\beta},  &\ff  n=0.
 \end{cases} \nonumber
 \end{eqnarray}
One has by calculation
($z=x+iy
\in \H$)
\begin{align*}
E(N, z, s) &=\frac{y^s}{2\zeta^{(N)}(2s)} \sum_{\substack{(m, n)\in \Z^2 \\ (N, n) =1}}  \frac{1}{|m N z+ n|^{2s}}
\\
&=y^s +  \frac{y^s}{N^{2s}\zeta^{(N)}(2s)} \sum_{m=1}^\infty \sum_{\substack{1 \le  a<N \\ (a, N)=1}} \sum_{j \in \Z} | mz + \frac{a}N + j|^{-2s}
\\
&=y^s +  \frac{y^s}{N^{2s}\zeta^{(N)}(2s)} \sum_{n \in \Z}   \sum_{m=1}^\infty t_n(my,s, s) \sum_{ a \in (\Z/N)^\times} e(\frac{n(mN x+ a)}N).
\end{align*}
Write
$$
\mathcal E(N, z, s)= N^{2s}\pi^{-s}\Gamma(s)\zeta^{(N)}(2s)E(N, z, s)=\sum_{k \in \Z} a_k(z, s) e(kx).
$$
Then we have
$$
a_{0}(z, s)=N^{2s}\pi^{-s}\Gamma(s)\zeta^{(N)}(2s)y^{s}+\varphi(N)y^{s}\pi^{-s} \frac{(2\pi)^{2s}\Gamma(2s-1)(4\pi
y)^{1-2s}\zeta(2s-1)}{\Gamma(s)}.
$$
Simple calculation gives
\begin{equation}\label{coefficient}
a_{0}(z, s)=\varphi(N)\bigg(\frac{1}{2(s-1)}-\frac{\log y}{2}- \frac{\log 4\pi -\gamma }{2} + \frac{\pi}6 y N\prod_{p|N}(1+p^{-1}) \bigg)
+O(s-1).
\end{equation}
On the other hand,
\begin{equation}\label{laurent}
\zeta^{\ast}(2s-1)= \frac{1}{2(s-1)} -\frac{1}{2}\big(\log 4\pi -\gamma\big) +O(s-1).
 \end{equation}
 So
 \begin{equation} \label{eq:constant}
 \lim_{s\rightarrow 1}(a_0(z, s) - \varphi(N) \zeta^{\ast}(2s-1))= \varphi(N) ( -\frac{\log y}{2}+\frac{\pi}6 y N\prod_{p|N}(1+p^{-1})).
 \end{equation}

For $k>0$, one has
\begin{eqnarray}
a_{k}(z, s)&=&y^{s}\pi^{-s}\Gamma(s)\sum_{mn=k}t_{n}(my, s, s)\sum_{a=1, (a, N)=1}^{N}e(nmx+anN^{-1}) \nonumber\\
&=&y^{s}\pi ^{-s}\Gamma(s)(2\pi)^{2s}\frac{W(4\pi ky, s, s )}{\Gamma(s)e^{2 \pi k y} }\sum_{mn=k}n^{2s-1}C_N(n).\nonumber
\end{eqnarray}
As
$$
W(4 k \pi y, 1,1)=\frac{1}{4 k \pi y},
$$
 one has
\begin{equation}\label{coefficient1}
a_{k}(z, 1)=\frac{e^{-2 \pi ky}}{k}\sum_{n\mid k}nC_{N}(n).
\end{equation}
 It is easy to see from definition that  $a_{-k}(z, 1)=a_{k}(z, 1)$.
Therefore,
\begin{align*}
&\mathcal E(N, z, s) =\sum_{k=-\infty}^{\infty}a_{k}(z, s)e(kx)
\\
&=a_0(z, s) +\sum_{k>0}\frac{1}{k}\sum_{n\mid k}nC_{N}(n)q_z^{k}+\sum_{k>0}\frac{1}{k}\sum_{n\mid k}nC_{N}(n)\overline{q_z^{k}}+O(s-1)
\\
&=a_0(z, s) + \sum_{n=1}^\infty C_N(n) \sum_{m=1}^\infty \frac{1}m(q_z^{mn} +\bar{q_z}^{mn})  +O(s-1).
\end{align*}
%where $q_z^{\prime}=e^{-2\pi i\bar{z}}.$ ({\bf it's $\bar q_z=q_z^{-2 \pi i \bar z}$})
Combining this with  (\ref{eq:constant}) and Proposition \ref{prop: DeltaN}, we obtain
\begin{align*}
&\lim _{s\rightarrow 1} \bigg(\mathcal E(N, z, s)- \varphi(N)\zeta^{\ast}(2s-1)\bigg)
\\
&=-\frac{\varphi(N)}2 \log y +\frac{ N \varphi(N) \prod_{p|N} (1+p^{-1}) \pi y}{6} -\sum_{n=1}^\infty \log|1-q_z^n|^2
\\
&=-\frac{1}{12} \log \big( y^{6 \varphi(N)} \mid\Delta_{N}(z)\mid\big),
\end{align*}
as claimed.
The second one follows from this identity immediately by applying $w_N$ on both sides.

{\bf Proof of Theorem \ref{derivatives}}:
%\begin{proof}
The first identity is just restatement of Corollary \ref{cor:Value}.  For the second identity,  we have by Theorems  \ref{firstlimitformula},   \ref{mainresult1} and Corollary \ref{cor:Value}
\begin{eqnarray}
&&-\frac{1}{12}{I}(\tau, \log  \mid\Delta_N(z) y^{6 \varphi(N)}\mid) \nonumber\\
&=&\lim _{s\rightarrow 1} \bigg({I}(\tau, \mathcal{E}(N,z,s)) -{I}(\tau, \varphi(N)\zeta^{\ast}(2s-1))\bigg)\nonumber\\
&=&\lim _{s\rightarrow 1}\bigg( \zeta^{\ast}(s)\mathcal{E}_{L}(\tau, s) -2\zeta^{\ast}(2s-1)\mathcal{E}_{L}(\tau,1)\bigg).\nonumber
\end{eqnarray}
Now the second identity for $\log\| \Delta_N(z)\|$ follows from elementary calculation of the Laurent  expansion (just first two terms) of  the functions in  the above expression. We leave the detail to the reader.
\begin{comment}
Recall the equation (\ref{laurent}) $$\zeta^{*}(s)=\frac{1}{s-1}-\frac{1}{2}(\log 4\pi-\gamma)+O(s-1),$$  one has $$\zeta^{\ast}(s)\mathcal{E}_{L}(\tau, s)=\frac{\mathcal{E}_{L}(\tau,1)}{s-1}+\mathcal{E}_{L}^{\prime}(\tau,1)-\frac{1}{2}(\log 4\pi-\gamma)\mathcal{E}_{L}(\tau,1)+O(s-1)$$ and
$$ 2\zeta^{\ast}(2s-1)\mathcal{E}_{L}(\tau,1)=\frac{\mathcal{E}_{L}(\tau,1)}{s-1}-(\log 4\pi-\gamma)\mathcal{E}_{L}(\tau,1)+O(s-1).$$
Thus $$-\frac{1}{12}{I}(\tau, \log  \mid \Delta_{N} y^{6\varphi(N)}\mid)=\mathcal{E}_{L}^{\prime}(\tau,1)+\frac{1}{2}(\log 4\pi- \gamma)\mathcal{E}_{L}(\tau,1).$$
So  \begin{eqnarray}
\mathcal{E}_{L}^{\prime}(\tau,1)&=&-\frac{1}{12}{I}(\tau, \log  \mid \Delta_{N} y^{6\varphi(N)}\mid +\frac{k}{2}\frac{\log 4\pi-\gamma}{2})\nonumber\\
&=&-\frac{1}{12}{I}(\tau, \log \parallel \Delta_{N} \parallel). \nonumber
\end{eqnarray}
One obtain the second identity. The proof for the third identity is the same as the second one and is left to the reader.
\end{comment}
%\end{proof}

\begin{proposition} \label{prop:CuspBehavoir} (1) \quad The generalized Delta function $\Delta_N(z)$  of level $N$  vanishes at the cusp $\infty$
 with vanishing order $N\varphi(N)\prod_{p \mid N}(1+p^{-1})$, and does not vanish  at other cusps.

  (2) \quad \begin{equation}
\Delta_N^0(z) =\Delta_N(z)|w_N = C_N \prod_{t|N} \Delta(tz)^{a(\frac{N}t)} \in M_k(N)
\end{equation}
 has vanishing order $\varphi(N) N \prod_{p|N} (1+p^{-1})$ at the cusp $P_0$ and does not vanish at other cusps. Here $C_N$ is the constant given in Proposition \ref{prop:Atkin-Lehner}.
\end{proposition}
\begin{proof} This proposition is clear  at cusp $P_{Q/N}$ with $Q\| N$ by Proposition \ref{prop:Atkin-Lehner}.  In  particular, it is true when $N$ is square free, which  is all what we need in  Part II.  The general case follows from the Kronecker limit formula at the cusp $P$.
 Write
$$
 N^{2s}\pi^{-s}\Gamma(s)\zeta^{(N)}(2s)=A+B(s-1)+o(s-1),
 $$
 and $\alpha =\frac{ \varphi(N)}{A}$.  According to \cite[(21)]{Go}, for a cusp $P$, there is $\sigma=\sigma_P \in \SL_2(\R)$ such that $\sigma(P_\infty) =P$, and
 \begin{eqnarray}&&\lim_{s\rightarrow 1}\bigg( E(N, \sigma z, s)- \frac{\alpha}{2(s-1)}\bigg) \nonumber\\
&&=\beta_P - \frac{\alpha}{2}\log y + y\delta_{P, P_\infty}+\sum_{m>1}(\phi_{P, m}q_z^{m}+\overline{\phi_{P, m} q_z^m}),\nonumber
\end{eqnarray}
for some constant $\beta_P$. Here $\delta_{P,P_\infty}$ is the Kronecker $\delta$-symbol.  So simple calculation gives for $P \ne P_\infty$
\begin{align*}
&\lim _{s\rightarrow 1} \bigg( \mathcal{E}(N, \sigma z, s)-\varphi(N)\zeta^{\ast}(2s-1)\bigg)
 \\
&=\gamma_P -\frac{\varphi(N)}2 \log y  + A \sum_{m>1}(\phi_{P, m}q_z^{m}+\overline{\phi_{P, m} q_z^m}),
\end{align*}
for some constant $\gamma_P$. One has thus by Theorem \ref{firstlimitformula}
$$
\log \big(y^{6\varphi(N)}\mid \Delta_{N}(\sigma(z))\mid\big)
=-12 \gamma_P + 6 \varphi(N) \log y - 12 A \sum_{m>1}(\phi_{P, m}q_z^{m}+\overline{\phi_{P, m} q_z^m}).
$$
Equivalently,
$$
\log|\Delta_N(\sigma(z))|= -12 \gamma_P - 12 A \sum_{m>1}(\phi_{P, m}q_z^{m}+\overline{\phi_{P, m} q_z^m}),
$$
which goes to $-12 \gamma_P$ when $y\rightarrow \infty$.
So $\Delta_N(z)$ does not vanish at the cusp $P=\sigma(P_\infty)$.

\end{proof}

Recall that the Eisenstein series $E(N, z, s)$ has the Fourier expansion
 $$E(N, z, s)=\sum_{n \in \Z}c_{n}(y, s)e(nx),$$
 where the constant term has the form
 $$
c_{0}(y, s)=y^{s}+\Phi(s)y^{1-s},
$$
with
\begin{equation}
\Phi(s)= \frac{\varphi(N)\pi^{\frac{1}{2}}\zeta(2s-1)\Gamma(s-\frac{1}{2})}{N^{2s}\zeta^{(N)}(2s)}.
\end{equation}
Simple calculation gives  the following lemma, which will be used  in the proof of Theorem \ref{theo:horizontal}.
\begin{lemma}\label{scatteringconst} Write
$$
\Phi(s) = \frac{C_{-1}}{s-1} + C_0 + O(s-1).
$$
Then
\begin{align*}
C_{-1} &= \hbox{Res}_{s=1} \Phi(s) =\frac{3}{\pi r},
\\
C_0 &=-\frac{6}{\pi r}\bigg( \log 4\pi-1 +12\zeta^{\prime}(-1)+ \sum_{p \mid N} \frac{p^{2}}{p^{2}-1} \log p\bigg),
\end{align*}
where $r=[\SL_{2}(\Z): \Gamma_{0}(N)]=N\Pi_{p \mid N}(1+p^{-1}).$
\end{lemma}
We remark that $C_0$ is the so-called scattering constant $C_{P_\infty,P_\infty}$ in  \cite{Kuhn}.

\part{Arithmetic intersection and derivative of Eisenstein series}

In this part, we will focus on the arithmetic intersection on the modular curve $X_0(N)$  and prove Theorems \ref{maintheo} and \ref{theo:modularity}.  We will assume from now on that $N$ is square free.
%First we briefly review the arithmetic intersection theory we need later in Section \ref{sect:ArithIntersectionReview} and refer to \cite{Kuhn}, \cite{BBK}, and %\cite{BKK} for detail.

\section{Metrized line bundles with log singularity and arithmetic divisors with log-log singularities}\label{sect:ArithIntersectionReview}
The Gillet-Soul\'e height pairing (see \cite{SouleBook}) has been extended to arithmetic divisors with log-log singularities  or equivalently metrized bundles with log singularities (\cite{BKK}, \cite{Kuhn}, \cite{Kuhn2}). It is also extended to arithmetic divisors with $L_1^2$-Green functions (\cite{Bost}).  In this paper, we will use K\"uhn's set-up in \cite{Kuhn}, which is most convenient in our situation. Actually, for simplicity, we use a stronger condition which is easier to state and enough for our purpose.

Let $\mathcal X$ be a regular  and proper stack over $\Z$ of dimension $2$ (called arithmetic surface), and denote $X=\mathcal{X}(\C)$. For a finite subset $S=\{ S_1, \cdots, S_r\}$ of $X$, let $Y=X-S$ be its complement.  For $\epsilon >0$, let $B_\epsilon(S_j)$ be the open disc of radius $\epsilon$  centered at $S_j$, and $X_\epsilon= X- \bigcup_{j} B_\epsilon(S_j)$.  Let $t_j$ be a local parameter at $S_j$. A metrized line bundle $\widehat{\mathcal L}=(\mathcal L,  \| \, \|)$ with log singularity (with respect to $S$) is a line bundle $\mathcal L$ over $\mathcal X$ together with a metric $\|\, \|$ on $\mathcal L_\infty=\mathcal L(\C)$ satisfying the following conditions:

\begin{enumerate}
\item  $\|\, \|$ is a smooth  Hermitian metric on $\mathcal L_\infty$ when restricting to $Y$.

\item  For each $S_j \in S$ and a (non-trivial) section $s$ of $\mathcal L$, there is a real number $\alpha_j$ and a positive smooth function $\varphi$ on $B_\epsilon(S_j)$ such that
    $$
    \| s(t_j)\| = (- \log|t_j|^2)^{\alpha_j} |t_j|^{\ord_{S_j}(s)} \varphi(t_j)
    $$
hold for all $t_j \in B_\epsilon(S_j)-\{0\}$ (here $t_j=0$ corresponds to $S_j$).
% Some further conditions on derivatives of $\varphi $ are required, which we refer to \cite{Kuhn} for details.
\end{enumerate}
Notice that $\widehat{\mathcal L}$ with log singularity is a regular metrized line bundle if and only if all $\alpha_j=0$.
 \begin{comment}
 Given a rational section $s$ of $\mathcal L$,  $\widehat{\Div}(s) = (\Div (s),  -\log\|s\|^2)$ is an arithmetic divisor with log-log singularity in the sense of \cite{BKK}(called pre-log-log there). Conversely, for an arithmetic divisor with log-log singularity (in particular a regular arithmetic divisor like $\widehat{\mathcal {Z}}$) $\widehat{\mathcal Z} =(\mathcal Z, g)$, there is an associated metrized line bundle $\widehat{\mathcal L}$ with log singularity  together with a canonical section $s$ with $\widehat{\Div}(s) =\widehat{\mathcal Z}$.
\end{comment}
We will denote $\widehat{\Pic}_\R(\mathcal X, S)$ for the group of metrized line bundles with log singularity (with respect to $S$) with $\R$-coefficients (i.e. allowing formally $\widehat{\mathcal L}^c$ with $c \in \R$).

 %and $\widehat{\CH}_\R^1(\mathcal X, S)$ for the arithmetic Chow group of arithmetic divisors with log-log singularity along $S$ with $\R$-coefficients.

 A pair $\widehat{\mathcal Z}=(\mathcal Z, g)$ is called an arithmetic divisor with  log-log-singularity (along $S$) if $\mathcal Z$ is a divisor of $\mathcal X$, and $g$ is a smooth function away from $Z\cup S$ ($Z=\mathcal Z(\C)$), and satisfying the following conditions:
 \begin{align*}
 dd^c g& + \delta_Z =[\omega],
 \\
 g(t_j) = -2 \alpha_j \log \log(\frac{1}{|t_j|^2})&-2 \beta_j \log |t_j| -2 \psi_j(t_j)  \quad \hbox{near  } S_j,
 \end{align*}
 for some smooth function  $\psi_j$ and some $(1, 1)$-form $\omega$ which is smooth away from $S$.
 Let $\widehat{\mathcal L}$ be the metrized line bundle associated to $\widehat{\mathcal Z}$ with canonical section $s$ with  $-\log\| s\|^2= g$, then  $\widehat{\mathcal Z}$ is of log-log-singularity if and only if $\widehat{\mathcal L}$ has log-growth and
 \begin{equation}
 \alpha_j(g) =\alpha_j(s), \quad \beta_j(g) =\ord_{S_j}(s),  \quad \psi_j(t_j) = \log \varphi (t_j).
 \end{equation}
 We define $\widehat{\CH}_\R^1(\mathcal X, S)$ be the quotient of the $\R$-linear combination of the
 arithmetic divisors of $\mathcal X$ with log-log growth along $S$ by $\R$-linear combinations of the principal arithmetic divisors with log-log growth along $S$. One has $\widehat{\Pic}_\R(\mathcal X, S) \cong \widehat{\CH}_\R^1(\mathcal X, S)$.
\begin{comment}
In  particular, $\widehat{\mathcal Z}(n, \mu, v)  \in \widehat{\CH}_\R^1(\mathcal X_0(N))$ (with respect to $S=\{\hbox{cusps}\}$ for all $n$, including $n\le 0$ with $D=-4Nn$ being a square.  Therefore, our generating function $\widehat{\phi}(\tau)$ has values in $\widehat{\CH}_\R^1(\mathcal X_0(N))$.
 \end{comment}
 The following is a \cite[Proposition 1.4]{Kuhn}.

\begin{proposition}  \label{prop:Intersection} There is an extension of the Gillet-Soul\'e height paring to
$$
\widehat{\CH}_\R^1(\mathcal X, S) \times \widehat{\CH}_\R^1(\mathcal X, S) \rightarrow \R
$$
such that if $\mathcal Z_1$ and $\mathcal Z_2$ are   divisors  intersect properly, then
$$
\langle (\mathcal Z_1, g_1),  (\mathcal Z_2, g_2)  \rangle
=(\mathcal Z_1.\mathcal Z_2)_{fin} + \frac{1}2 g_1*g_2
$$
where the star product is defined to be
\begin{align*}
g_1*g_2&= g_1( Z_2 -\sum_{j} \ord_{S_j} (Z_2) S_j)
 + 2 \sum_{j} \ord_{S_j} Z_2 \left( \alpha_{\mathcal L_1, j} -\psi_{1, j}(0)\right)
 \\
&\quad  -\lim_{\epsilon \rightarrow 0} \left( 2 \sum_j (\ord_{S_j}Z_2) \alpha_{\mathcal L_1, j} \log(-2 \log \epsilon)-\int_{X_\epsilon} g_2 \omega_1\right).
\end{align*}

Here  $Z_i=\mathcal Z_i(\C)$, $\widehat{\mathcal L}_i$ is the associated metrized line bundle with the canonical section $s_i$. $\alpha_{\mathcal L_i, j}$ and $\psi_{i, j}$ are associated to $g_i$ and  cusp $S_j$.  Finally, $\omega_i$ is the $(1, 1)$-form associate to $g_i$ via  the following equation
$$
d d^c [g_i] + \delta_{Z_i} = [\omega_i].
$$
\end{proposition}

We remark  that the pairing is also symmetric.  In particular, one  has for  any $a(f) = (0, f) \in  \widehat{\CH}_\R^1(\mathcal X, S)$,
\begin{equation} \label{eq:VertInfiniteIntersection}
\langle (\mathcal Z, g), a(f)  \rangle = \frac{1}2 \int_X f \omega.
\end{equation}
We define the degree map
\begin{equation} \label{eq:Degree}
\deg:  \widehat{\CH}_\R^1(\mathcal X, S)\rightarrow \R, \quad \deg (\mathcal Z,g) = \int_{X} \omega = \langle (\mathcal Z,g), (0, 2) \rangle.
\end{equation}
  It is just $\deg Z$ when $g$ is a Green function of $Z=\mathcal Z(\C)$ without log-log singularity.

  We will denote $\widehat{\CH}_\R^1(\mathcal X)=  \widehat{\CH}_\R^1(\mathcal X, \hbox{empty})$ for the usual arithmetic Gillet-Soul\'e Chow group with real coefficients.

\section{Kudla's  Green function} \label{sect:KudlaGreenFunction}

Let $V=\{w \in M_2(\Q):\, \tr(w)=0\}$ be the quadratic space with quadratic form $Q(w) =N \det w$, and let $\mathbb D$  be the associated Hermitian symmetric domain of positive lines in $V_\R$ as in  Section \ref{sect:ThetaLifting}. Recall that  $\SL_2 =\Spin(V)$ acts on $\D$ by conjugation, and $\D$ can be  identified with
 $\H$  (Lemma \ref{lem2.1}) via
\begin{equation}
w(z) =\frac{1}{\sqrt N y} \kzxz {-x} {z \bar z} {-1} {x}, \quad z =x+ i y \in \H.
\end{equation}
%This identification is $\SL_2$-equivariant.

Let $L$ be an even integral lattice with dual lattice $L^\sharp$ (arbitrary in this section). Let $\Gamma \subseteq \SL_2(\Z)$ be a subgroup of finite index which fixes $L$ and acts on $L^\sharp/L$ trivially. We denote $\bar{\Gamma} =\Gamma/(\Gamma \cap \{ \pm 1\})$.  For each pair $(n, \mu) \in \Q \times L^\sharp/L$ with $n>0$, $Q(\mu) \equiv n \pmod 1$,
 let $Z(n, \mu)$ be the associated Heegner divisor given by
$$
Z(n, \mu)=\Gamma \backslash \{ \R w: w \in L_\mu[n]\}.
$$
Kudla defined a nice Green function for $Z(n, \mu)$ in his seminal work \cite{Kucentral}, which we now briefly review. The purpose of this section is to understand its behavior at the cusps.

For $r>0$ and $s \in \R$, let
\begin{equation}\label{belta}
\beta_s(r) =\int_1^\infty e^{-rt } t^{-s} dt
\end{equation}
and
\begin{comment}
For  $w =\left(
  \begin{array}{cc}
    w_{1}&w_{2}\\
     w_{3}&-w_{1}\\
  \end{array}
\right) \in V_\R,$ $z=x+iy \in \H$, recall
\begin{equation}\label{formula1}
(w, w(z))=-\frac{\sqrt{N}}{y}(w_3z\overline{z}-w_1(z+\overline{z})-w_2).
\end{equation}
Decompose
$$
V_\R =\R w(z) \oplus w(z)^\perp, \quad w = w_z + w_{z^\perp},
$$
and define
$R(w, z) =-(w_{z^\perp}, w_{z^\perp})$, and the majorant
$$
(w, w)_z= (w_z, w_z) + R(w, z)= (w_z, w_z) - (w_{z^\perp}, w_{z^\perp}).
$$
\end{comment}
\begin{equation}
\xi(w, z)=\beta_1(2\pi R(w, z)),
\end{equation}
be Kudla's $\xi$-function.
For $\mu \in L^\sharp/L$, $n \in Q(\mu)+\Z$ and $v \in \R_{>0}$, define
\begin{equation}\label{eq:KudlaGreenFunctionDefinition}
\Xi(n,\mu,  v)(z) =\sum_{0\ne w \in L_\mu[n]}\xi (v^{\frac{1}{2}}w, z).
\end{equation}
Then Kudla has proved on $Y_0(N)$ (\cite{Kucentral}) that $\Xi(n, \mu, v)$ is a Green function for $Z(n, \mu)$ and satisfies the following  Green current equation:
$$
d d^c  [\Xi(n,\mu,  v)] +\delta_{Z(n, \mu)}=[\omega(n, \mu, v)]
$$
when $n>0$. When  $n \le 0$,  $\Xi(n, \mu, v)$ is still well-defined and actually  smooth on  $Y_0(N)$ while $Z(n, \mu)=0$.  So $\Xi(n, \mu, v)$ is  a Green function for $Z(n, \mu)$ for all $n$.  The purpose of this section is understand its behavior at cusps, which is quite complicated and subtle.

Let $\Iso(V)$ be the set of isotropic non-zero vectors of $V$, i.e., $0\ne \ell \in V$ with $Q(\ell)=0$. Given $\ell =\abcd \in \Iso(V)$, let $P_\ell= \frac{a}c$ be the associated cusp, which depends only on the isotropic line $\Q\ell$. Two isotropic lines give the same cusp in $\Gamma \backslash \H$ if and only if there is $\gamma \in \Gamma$ such that $\Q \gamma\cdot \ell_1 =\Q \ell_2$.

\begin{comment}
Let $\Gamma \subseteq \SL_{2}(\Z)$ be a subgroup of finite index, $\bar\Gamma =\Gamma/(\Gamma \cap \{\pm 1\})$ and  $\bar\Gamma_\infty = \{\kzxz {1} {\kappa \Z} {0} {1}\}$ be the stabilizer of the cusp $P_\infty$ (also  the stabilizer of the isotropic vector $\ell_\infty =\kzxz {0} {1} {0} {0}$), where $\kappa$ is the width of the modular curve $X_\Gamma$ at the infinite cusp $P_\infty$.  Notice that for a vector $w \in V$, $(w, \ell_\infty)=0$ if and only if $w=m \kzxz {1} {2r} {0} {-1}$ for some rational numbers $m$ and $r$. Fixing $m \ne 0$, two such elements $w_1$ and $w_2$ are $\Gamma$-equivalent if and only if $r_1 -r_2 \in \kappa \Z$.
\end{comment}
Let $\ell_\infty=\kzxz {0} {1} {0} {0} \in \Iso(V)$ and let $P_\infty=\infty$ be its associated cusp. In general,
for an  isotropic element $\ell $, there exists $\sigma_{\ell} \in \SL_{2}(\Z)$ such that $\Q \sigma_{\ell}\cdot\ell_\infty= \Q \ell.$
  Then
  $$\sigma_{\ell}^{-1}\Gamma_{\ell}\sigma_{\ell}=\{ \pm\kzxz {1}{m\kappa_{\ell}}{0}{1}, m\in \Z\},$$
   where $\Gamma_{\ell}$ is the stabilizer of $\ell$ and $\kappa_{\ell} > 0$ is the classical width of the associated  cusp $P_{\ell}$, and $q_\ell  $ is a local parameter at the cusp $P_\ell$.
   On the other hand, there is  another positive number $\beta_{\ell}> 0$, depending on $L$ and the cusp $P_\ell$,  such that $\kzxz{0}{\beta_{\ell}}{0}{0}$ is a primitive element in $\Q\ell_\infty \bigcap \sigma_{\ell}^{-1}\cdot L$.  We denote $\varepsilon_{\ell}=\frac{\kappa_{\ell}}{\beta_{\ell}}$ and call it Funke constant at cusp $P_\ell$ although Funke called it width at $P_\ell$ in \cite[Section 3]{Fu}. We will simply denote $\kappa =\kappa_\infty$.

The main purpose of this section is  to prove the following technical theorem.

\begin{theorem} \label{theo:Singularity} Let the notation be as above.  Let $0\ne \ell \in \Iso(V)$ be an isotropic vector and  $P_\ell$ be the associated cusp.

 \begin{enumerate}

 \item  When $D=-4n N$ is not a square, $\Xi(n, \mu, v)$  is smooth and of exponential decay at the cusp $P_\ell$.

\item  When  $D=-4nN >0$ is a square.  Then  $\Xi(n, \mu, v)$ has log singularity at the cusp $P_\ell$ with
$$
\Xi(n, \mu, v) = - g(n, \mu, v, P_\ell) (\log|q_\ell|^2)
    -2\psi_\ell(n, \mu, v; q_\ell).
$$
Here  $q_\ell$ is a local parameter at the cusp $P_\ell$,
$$
\alpha_\Gamma(n, \mu, P_\ell)=\sum_{w \in L_\mu[n] \mod \Gamma}  \delta_{w},
$$
where $0\le \delta_w \le 2$ is the number of  isotropic lines $\Q \ell_w \in \Iso(V)$  which is perpendicular to $w$ and belongs to the same cusp as $\ell$, and $$
g(n,\mu, v, P_\ell)= \frac{1}{8 \pi  \sqrt{-nv}}\beta_{3/2}(-4 nv \pi) \alpha_\Gamma(n, \mu, P_\ell).
$$
Finally,  $\psi_\ell(n, \mu, v; q_\ell)$ is a smooth function of $q_\ell$ (as two real  variables $q_\ell$ and $\bar{q_\ell}$) and
$$
\lim_{q_\ell \rightarrow 0} \psi_\ell(n, \mu, v; q_\ell) =0.
$$

\item  When $D=0$, one has
 \begin{eqnarray}
\Xi(0, \mu, v)&=& - g(0, \mu, v, P_\ell) (\log|q_\ell|^2)-2 \log(-\log|q_{\ell}|^2)
    \nonumber\\
& &-2 \psi_\ell(0, \mu, v;q_\ell),\nonumber
\end{eqnarray}
where $q_\ell$ is the local parameter at $P_{\ell}$ with respect to the classical width $\kappa_{\ell}$,
$g(0, \mu, v, P_\ell) = \frac{\varepsilon_{\ell}}{2 \pi \sqrt{vN}}$.
Here $\varepsilon_{\ell}$ is the Funke constant of $\ell$.
Finally,  $\psi_\ell(0, \mu, v; q_\ell)$ is a smooth function of $q_\ell$ (as two real  variables $q_\ell$ and $\bar{q_\ell}$) and
$$
\lim_{q_\ell \rightarrow 0} \psi_\ell(0, \mu, v; q_\ell) =\begin{cases}
   \log\frac{\varepsilon_\ell}{4 \pi \sqrt{Nv}} -\frac{1}{2} f(0)    &\ff \mu \in L,
   \\
    \frac{1}{2}\log\frac{\varepsilon_\ell^2}{4Nv \pi^3} +\frac{\gamma_1(0)}{2} - \sum_{n=1}^\infty \frac{\cos(\frac{2 \pi n \mu_\ell}{\beta_\ell})}n &\ff \mu \notin L.
    \end{cases}
$$
Here $f(0) = \gamma -\log(4\pi)$ is defined in Lemma \ref{lem:beta},
$$
\gamma_1(0) =\int_1^\infty e^{-y} \frac{dy}y + \int_0^1 \frac{e^{-y}-1}y dy
$$
and
$$
\sigma_\ell^{-1}\cdot L_\mu \cap \Q \ell =  \{ \kzxz  {0} {\mu_\ell + m \beta_\ell} {0} {0}:\, m \in \Z\}.
$$
\end{enumerate}
\end{theorem}

The proof is long and technical and will occupy the next few subsections.

\begin{comment}
\item  When $D=0$,  $\mu =0$, \begin{eqnarray}
\Xi(0, 0, v)&=& - g(0, 0, v, P_\ell) (\log|q_\ell|^2)-2 \log(-\log|q_{\ell}|^2)
    \nonumber\\
& &+ \hbox{ smooth},\nonumber
\end{eqnarray}
where $q_\ell$ is the local parameter at $P_{\ell}$ with respect to the classical width $\alpha_{\ell}$,
$g(0, 0, v, P_\ell) = \frac{\varepsilon_{\ell}}{2 \pi \sqrt{vN}}$. Here $\varepsilon_{\ell}$ is the width of $\ell$ ; when $\mu \neq 0$, $\Q \ell \cap L_\mu[0]$  is empty, then  $\Xi(0, \mu, v) =0.$
\end{enumerate}
We will give the exactly formulas of $g(n, \mu, v, P_\ell) $ and  $\varepsilon_{\ell}$ in the later part of this section.
\end{comment}

\subsection{ Two lemmas}

\begin{lemma}  \label{lem:beta} Let $a >0$ and $z =x +iy \in \C$. Then
\begin{enumerate}
\item  When $z \notin \R$, one has
$$
\sum_{n \in \Z} \beta_1(\pi a^2|z+n|^2) =\frac{1}a \sum_{n \in \Z} e(nx) \int_1^\infty e^{-\pi a^2 y^2 t -\frac{\pi n^2}{a^2t } } t^{-\frac{3}2} dt.
$$

\item When  $z =x \in \R-\Z$, one has
$$
\sum_{ n \in \Z} \beta_1(\pi a^2 (x+n)^2)=2\sum_{n \in  \Z} e(nx)\int_0^{\frac{1}a}e^{-\pi n^2 t^2} dt.
$$
Moreover, one has near $a=0$
$$
\sum_{ n \in \Z} \beta_1(\pi a^2 (x+n)^2)= \frac{2}a + f(a, x),
$$
for some smooth function $f(a, x)$ near $a=0$ with
$$
f(0, x) = \lim_{a \rightarrow 0} f(a, x)= 2\sum_{n=1}^\infty \frac{\cos(2\pi n x)}{n}.
$$

\item  One has
$$
\sum_{0 \ne n \in \Z} \beta_1(\pi a^2 n^2) = 2\int_{0}^{\frac{1}a}  \left( \sum_{n\in \Z} e^{-\pi n^2 t^2} -\int_\R e^{-\pi x^2 t^2} dx \right) dt.
$$
Moreover,
One has
  near $a =0$
$$
\sum_{0\ne  n \in \Z} \beta_1(\pi a^2 n^2)= \frac{2}a +2 \log a + f(a),
$$
for some smooth function $f(a)$ near $a=0$  with
$$
f(0)= \lim_{a \rightarrow 0} f(a)
    =\gamma -\log(4\pi),
$$
where $\gamma$ is the Euler constant.

\end{enumerate}
\end{lemma}
\begin{proof} Let
$$f(n) = \beta_1(\pi a^2 |z+n|^2) = \beta_1( \pi a^2 y^2 +\pi a^2 (x+n)^2).
$$
Then its Fourier transformation is
 \begin{align*}
\widehat{f}(n)  &=\int_\R f(\alpha) e(-\alpha n) d\alpha
%\\
%&=\int_1^\infty e^{-\pi a^2 y^2 t} \frac{dt}t \int_\R e^{-\pi a^2 (x+\alpha)^2t  - 2 \pi i n \alpha}  d\alpha
\\
 &= \frac{e(n x)}{a} \int_1^\infty  e^{-\pi a^2 y^2 t - \frac{\pi n^2}{a^2t} } t^{-\frac{3}{2}}\, dt.
 \end{align*}
 Now applying the Poisson summation formula, one obtains  the formula in (1).  When $y=0$, simple substitution gives  part of (2) with  $x \notin \Z$.
To see the behavior of the sum near $a=0$, notice that the right-hand  side is equal to
$
\frac{2}a +  f(a, x)
$
with
$$
f(a, x) = 2 \sum_{n=1}^\infty (e(nx) +e(-nx)) \int_0^{\frac{1}a} e^{-\pi n^2 t^2} dt.
$$
It is clearly smooth near $a=0$ if we define
$$
f(0, x)=2 \sum_{n=1}^\infty (e(nx) +e(-nx)) \int_0^{\infty} e^{-\pi n^2 t^2} dt =2 \sum_{n=1}^\infty \frac{\cos(2\pi n x)}{n}.
$$

To prove (3),
 take  $z =i\epsilon$ in (1), and let $\epsilon$ goes to zero, we obtain
  $$
  \sum_{0 \ne n \in \Z} \beta_1(\pi a^2 n^2)=\lim_{\epsilon \rightarrow 0}\left[ \frac{1}a \sum_{n \in \Z}  \int_1^\infty e^{-\pi a^2 \epsilon^2 t -\frac{\pi n^2}{a^2} t^{-1}} t^{-\frac{3}2} dt -\beta_1(\pi a^2 \epsilon^2)\right].
  $$
  By the Fourier inversion formula, one has
  $$
\frac{1 }a  \int_1^\infty \int_\R e^{-\pi a^2 \epsilon^2 t -\frac{\pi x^2}{a^2} t^{-1}} t^{-\frac{3}2} dx  dt= \beta_1(\pi a^2 \epsilon^2).
$$
  So
 \begin{align*}
  &\sum_{0 \ne n \in \Z} \beta_1(\pi a^2 n^2)
  \\
  &=\frac{1}a\lim_{\epsilon \rightarrow 0}\int_1^\infty e^{-\pi a^2 \epsilon^2 t}
  \left[  \sum_{  n \in \Z} e^{  -\frac{\pi n^2}{a^2t} } -\int_\R e^{  -\frac{\pi n^2}{a^2t} } dx \right] t^{-\frac{3}2} dt
  \\
  &=  \frac{2}a\int_1^\infty\left[  \sum_{ n \in \Z} e^{  -\frac{\pi n^2}{a^2t} } -\int_\R e^{  -\frac{\pi n^2}{a^2t} } dx \right] t^{-\frac{3}2} dt
  \\
  &=2 \int_{0}^{\frac{1}a} (\sum_{n \in \Z}   e^{-\pi n^2 t^2} -  \int_\R e^{-\pi x^2 t^2} dx ) dt
  \\
  &=\frac{2}a -4 \int_0^{\frac{1}a} \int_{0}^1 e^{-\pi x^2 t^2} dx  dt
    + 4\sum_{n=1}^\infty \left[\int_{0}^{\frac{1}a}e^{-\pi n^2 t^2}dt - \int_{0}^{\frac{1}a} \int_{n}^{n+1} e^{-\pi x^2 t^2} dx  dt \right]
    \\
    &=\frac{2}a -4 g_0(a) + 4 \sum_{n=1}^\infty g_n(a),
\end{align*}
with obvious meaning of $g_n(a)$.
  Here we have used the fact that the integrand in the last integral is negative. The term $\frac{2}a$ comes from  the term $n=0$  in the sum. We remark that the formula looks formally like ($z=0$)
  $$
  \sum_{n\ne 0} f(n) = \sum_{n \in \Z} \widehat{f}(n) -\int_\R \widehat f(x) dx.
  $$
  What we did is to regularize the right hand side to make it meaningful.

First,
\begin{align*}
g_0(a) &=\int_{0}^1 \int_0^{1} e^{-\pi x^2 t^2} dx  dt + \int_{1}^{\frac{1}a} \int_{0}^1 e^{-\pi x^2 t^2} dx  dt
\\
 &= \int_{0}^1 \int_0^{1} e^{-\pi x^2 t^2} dx  dt + \int_{1}^{\frac{1}a} \int_{0}^\infty  e^{-\pi x^2 t^2} dx  dt -\int_{1}^{\frac{1}a} \int_{1}^\infty e^{-\pi x^2 t^2} dx  dt
 \\
 &=-\frac{1}2 \log a + \int_{0}^1 \int_0^{1} e^{-\pi x^2 t^2} dx  dt-\int_{1}^{\frac{1}a} \int_{1}^\infty e^{-\pi x^2 t^2} dx  dt.
\end{align*}
Putting the integrals
$$
\int_{0}^1 \int_0^{1} e^{-\pi x^2 t^2} dx  dt-\int_{1}^{\infty} \int_{1}^\infty e^{-\pi x^2 t^2} dx  dt
$$
into Mathematica, it comes out the answer $\frac{1}4(\gamma+ \log 4\pi)$. So we have
$$
\lim_{a \rightarrow 0} (g_0(a) + \frac{1}2 \log a)
 = \frac{1}4(\gamma+ \log 4\pi).
$$
Next, we have
\begin{align*}
&\lim_{a \rightarrow 0} \sum_{n=1}^\infty \left[\int_{0}^{\frac{1}a}e^{-\pi n^2 t^2}dt - \int_{0}^{\frac{1}a} \int_{n}^{n+1} e^{-\pi x^2 t^2} dx  dt \right]
\\
 &= \sum_{n=1}^\infty \left[\int_{0}^{\infty }e^{-\pi n^2 t^2}dt - \int_{n}^{n+1}\int_{0}^{\infty }  e^{-\pi x^2 t^2}  dt dx \right]
 \\
  &=\frac{1}2 \sum_{n=1}^\infty (\frac{1}n-\log\frac{n+1}n)=\frac{1}2 \gamma.
\end{align*}
In  summary, we have
$$
\sum_{0\ne  n \in \Z} \beta_1(\pi a^2 n^2)= \frac{2}a +2 \log a + f(a),
$$
for some smooth function $f(a)$ near $a=0$  with
$$
f(0) =\lim_{a \rightarrow 0} f(a) =-(\gamma +\log(4\pi)) + 2 \gamma = \gamma -\log(4\pi).
$$

\end{proof}

\begin{lemma} \label{lem:infintegreen}
Assume that $D=-4Nn= (2Nm)^2>0$ is a square.  For any $w= w(m, r)= m\kzxz {1} {2r} {0} {-1} \in L_\mu[n]$
with  $(w, \ell_\infty) =0$, define
$$
\Xi_\infty(w, z) = \sum_{\gamma \in \bar\Gamma_{\infty}} \xi(w, \gamma z).
$$
Then for any $v>0$
$$
\Xi_\infty(\sqrt v w, z) =-(\log|q_\kappa|)  \frac{\sqrt N}{2 \pi \sqrt{Dv}}
   \sum_{n \in \Z} e(\frac{n}\kappa (x+r)) \int_{1}^\infty e^{-\left(\frac{tDv}N +  n^2 \frac{N}{tDv} \frac{y^2}{\kappa^2}\right)\pi } \frac{dt}{t^{\frac{3}2}},
$$
where $q_\kappa =e(z/\kappa)$ is a local parameter of $X_\Gamma$ at the cusp $P_\infty$.
Moreover, one has near the cusp $P_\infty$ ($q_\kappa =0$)
 $$
 \Xi_\infty(\sqrt v w, z)=-(\log|q_\kappa|^2)  \frac{\sqrt N}{4 \pi \sqrt{Dv}} \beta_{\frac{3}2}(\frac{D v\pi}N)+ f(\sqrt v w, z),
$$
where $f(\sqrt v w, z)$ is a smooth function of $x$ and $y$ near $P_\infty$ and
$$
\lim_{y \rightarrow \infty} f(\sqrt v w, z) =0.
$$
\end{lemma}

\begin{proof} One has
  $\bar\Gamma_\infty = \{\kzxz {1} {\kappa \Z} {0} {1}\}$
and
\begin{align*}
R(\sqrt v w, \kzxz {1}{n\kappa} {0} {1} z)&=\frac{v}{2}( w, w(z+n\kappa))^{2}-v (w, w)
\\
&=\frac{Dv}{2Ny^{2}}\mid z+n\kappa+r \mid^{2}.
\end{align*}
So one has by   Lemma \ref{lem:beta},
\begin{eqnarray}\label{greeninfinite}
 \Xi_\infty(\sqrt v w, z)
 &=& \sum_{n \in \Z} \beta_1(\frac{\pi D v}{N y^2} |z+ r + n\kappa|^2) \nonumber
 \\
 &=&  \frac{y \sqrt N}{\kappa \sqrt{Dv}}
   \sum_{n \in \Z} e(\frac{n}\kappa (x+r)) \int_{1}^\infty e^{-\left(\frac{tD}N +  n^2 \frac{N}{tD} \frac{y^2}{\kappa^2}\right)\pi} \frac{dt}{t^{\frac{3}2}}\nonumber\\
   &=&-(\log|q_\kappa|^2)  \frac{\sqrt N}{4 \pi \sqrt{Dv}} \beta_{\frac{3}2}(\frac{\pi D}N) + f(\sqrt v w, z)\nonumber
\end{eqnarray}
with
$$
f(\sqrt v w,z) =-(\log|q_\kappa|^2)  \frac{\sqrt N}{4 \pi \sqrt{Dv}}
   \sum_{0 \neq n \in \Z} e(\frac{n}\kappa (x+r)) \int_{1}^\infty e^{-\left(\frac{tDv}N +  n^2 \frac{N}{tDv} \frac{y^2}{\kappa^2}\right)\pi} \frac{dt}{t^{\frac{3}2}}.
$$
Since
$$
\frac{tDv}N +  n^2 \frac{N}{tDv} \frac{y^2}{\kappa^2} \ge \frac{2|n| y}{\kappa},
$$
one sees  for all $n\ne 0$
$$
\left| e(\frac{n}\kappa (x+r)) \int_{1}^\infty e^{-\left(\frac{tDv}N +  n^2 \frac{N}{tDv} \frac{y^2}{\kappa^2}\right)\pi} \frac{dt}{t^{\frac{3}2}} \right| \leq 2e^{-2\frac{\mid n\mid y}{\kappa}\pi},
$$
and
$$
|f(\sqrt v w, z)| \le   \frac{4 \sqrt N y}{  \kappa\sqrt{Dv} }\sum_{n=1}^\infty e^{-2\frac{n \pi y}{\kappa}}
$$
which is of exponential decay as $y \mapsto \infty$. This proves the lemma.
\end{proof}

\subsection{Proof of Theorem \ref{theo:Singularity}}

\begin{proof}  Now we are ready to start proof of Theorem \ref{theo:Singularity}. By transformation, we may assume that $\ell =\ell_\infty$ is associated to the cusp $P_\infty$. Then $q_\ell =q_\kappa$ where $\kappa$ is the width of the cusp $P_\infty$ as in the above lemma. We divide the proof into three steps: general set-up and the case $D=-4Nn$ is not a square,  $D>0$ being a square, and $D=0$.

{\bf Step 1: Set-up and the case that $D$ is not a square.}
 We write
\begin{equation}
\Xi(n, \mu, v) = \sum_{w \in  L_\mu[n] \mod \Gamma} \Xi(\sqrt v w, z), \quad   \Xi(\sqrt v w, z) =\sum_{\gamma \in \bar\Gamma_w \backslash \bar\Gamma} \xi(\sqrt v w, \gamma z).
\end{equation}
For $w=\kzxz {w_1} {w_2} {w_3} {-w_1} \in L_\mu[n]$, let ${\tilde w}= \kzxz {w_3} {-w_1} {-w_1} {-w_2}= S^{-1}\cdot w$ with $S= \kzxz {0} {-1} {1} {0}$. Then  ${\tilde w}$ is symmetric.
 Simple calculation gives for $\gamma =\abcd \in \Gamma$
\begin{equation}
R(w, \gamma z) = \frac{N}{2y^2} \left[ h_{{\tilde w}}(\gamma, z)\right]^2 -n,
\end{equation}
where
$$
h_{{\tilde w}} (\gamma, z) =(a z +b, c z+d) {\tilde w} \overline{(az+b, c z+d)}^t= Q_{{\tilde w}} (a, c) y^2 + Q_{{\tilde w}}( a x+ b, c x +d)
$$
is the  Hermitian form on $(\R z +\R)^2$, and $Q_{{\tilde w}}$ is the quadratic form on $\R^2$ associated to ${\tilde w}$. Notice that $\{(a z +b, cz+d):\, \gamma \in \Gamma\}$ is a subset of a lattice of $(\R z +\R)^2$, so for any positive number $M$
$$
\#\{ \gamma =\abcd \in \Gamma_\infty \backslash \Gamma: \, |h_{{\tilde w}}(\gamma,z)| \le M \quad \hbox{ and } \quad  0<| Q_{{\tilde w}}(a, c)| \le M\}
$$
are finite and of polynomial growth as functions of $M$. Moreover  there is a positive number $M_0$ such that if $Q_{{\tilde w}}(a, c)\ne 0$ for some $\gamma \in \Gamma$, then $|Q_{{\tilde w}}(a, c)| \ge M_0$.  In such a case,  we have
$$
R(w, \gamma z) \thicksim \frac{N}{2} Q_{\tilde w}(a, c)^2 y^2
$$
as $y \rightarrow \infty$.
Recall that
$$
\beta_1(t) =O(e^{-t}/t)
$$
as $t \rightarrow \infty$. Therefore the terms with $Q_{{\tilde w}}(a, c) \ne 0$ in the sum $\Xi(\sqrt v w, z)$ goes to zero in an exponential decay fashion. So we have proved the following lemma.

\begin{lemma} \label{lem6.4}  Let the notation be as above. If there is no $\gamma =\abcd  \in \Gamma$ such that $Q_{\tilde w}(a, c) =0$,  then
$
\Xi(\sqrt v w, z)
$
is smooth at the cusp $P_\infty$ and is of exponential decay as $y\rightarrow \infty$.
\end{lemma}

When $D$ is not a square, the quadratic form  $Q_{\tilde w}$  does not represent $0$. So $\Xi(n, \mu, v)$ is of exponential decay in this case when $y \rightarrow \infty$. This proves (1).

{\bf Step 2: Next, we assume $D=-4Nn >0$ is a square}. In this case, $\bar{\Gamma}_w=1$
$$
0=Q_{{\tilde w}}(a , c) = w_3 a^2 -2 w_1 ac - w_2 c^2
$$
has exactly two  integral solutions $(a_i, c_i) \in \Z^2$ such that $\hbox{gcd}(a_i, c_i)=1$, $a>0$ or $a_i=0, c_i=1$. So
$w^\perp\cap \Iso(V)$ consists exactly two cusps $\Q\ell_{a_i, c_i}$ where $\ell_{a, c}= \kzxz {ac} {-a^2} {c^2}  {-ac}$.

For a fixed solution $(a, c)$, if there is $\gamma_0=\kzxz {a} {b} {c} {d} \in \Gamma$, then the cusp $P_{\frac{a}c}$ (corresponding to $\Q \ell_{a, c}$) is $\Gamma$-equivalent to $P_\infty$:  $\gamma_0 P_\infty =P_{\frac{a}c}$, and all $\gamma=\kzxz {a} {*} {c} {*} \in \Gamma$ with  $Q_{\tilde w}(a, c)=0$ is of the form $\gamma_0 \gamma_1 $ with $\gamma_1 \in \Gamma_\infty$. Therefore the sum related to this solution $(a, c)$ is
\begin{align*}
&\sum_{\substack{\gamma =\kzxz {a} {*} {c} {*} \in \bar\Gamma \\ Q_{\tilde w}(a, c)=0}}
  \xi(\sqrt v w, \gamma z)
= \sum_{\gamma_1 \in  \bar{\Gamma}_\infty}  \xi(\sqrt v \gamma_0^{-1}\cdot w, \gamma_1 z)
\\
&=\Xi_\infty(\sqrt v \gamma_0^{-1}\cdot w, z)
\\
&= -(\log|q_\kappa|^2)  \frac{\sqrt N}{4 \pi \sqrt{Dv}} \beta_{\frac{3}2}(\frac{D v\pi}N)+ f(\sqrt v \gamma_0^{-1}\cdot w, z)
\end{align*}
by Lemma \ref{lem:infintegreen}.  Recall  $\lim_{y\rightarrow \infty }f(\sqrt  v \gamma_0^{-1}\cdot w, z)=0$ by Lemma \ref{lem:infintegreen}. So  we have  by Lemma \ref{lem6.4},
\begin{align*}
\Xi(\sqrt v w, z)
&=\sum_{\substack{Q_{\tilde w}(a, c)=0 \\ \hbox{gcd}(a, c)=1\\ a>0 \hbox{ or } a=0, c=1}} \sum_{\gamma=\kzxz {a} {*} {c} {*} \in \bar\Gamma} \xi(\sqrt v w, z)
 + \sum_{\substack{\gamma=\abcd \in \bar{\Gamma} \\ Q_{\tilde w}(a, c) \ne 0}} \xi(\sqrt v w, z)
 \\
 &= - \frac{\delta_w \sqrt N}{4 \pi \sqrt{Dv}} \beta_{\frac{3}2}(\frac{D v\pi}N)(\log|q_\kappa|^2)  + \psi(w, z)
\end{align*}
with $\psi(w, z)$ smooth at the cusp $P_\infty$ and
$$
\lim_{y\rightarrow \infty } \psi(w, z) =0.
$$
Combining  this with Lemma \ref{lem6.4}, we proved (3) of Theorem \ref{theo:Singularity}.

\begin{comment}
\begin{lemma} Let $0\ne w= \kzxz{w_1} {w_2} {w_3} {-w_1} \in L_\mu[n]$,  $\tilde w= \kzxz {w_3} {-w_1} {-w_1} {-w_2}$, and $Q_{\tilde w}$ be the associated quadratic form defined above. Let $\gamma =\abcd  \in \Gamma$. Then $Q_{\tilde w}(a, c) =0$ if and only if the cusp $P_{-\frac{a}c}$ is $\Gamma$-equivalent to $P_\infty$ and $D=-4Nn \ge 0$ is a square.
\end{lemma}
\begin{proof} A quadratic form if and only if $\det (\tilde w) =-\frac{n}N$ is a square, i.e., $D$ is a square.

 \end{proof}

 adds  to a smooth function near the cusp $P_\infty$.   In particular, this implies that  $\Xi(n, \mu, v)$ is smooth at $P_\infty$ (and thus all cusps)  if $D=-4nN$ is not a square. This proves (1).

with $\gamma_0^{-1} w$ perpendicular to $\ell_\infty$ and $\Gamma$-equivalent to $w$. Since  the singularity part of $\Xi_\infty(\sqrt v  w, z)$ does not depend on the choice of $w$, we proved the proposition.
\end{comment}

{\bf Step 3: Finally we assume $n=0$}.  Each vector $0\ne w \in L_\mu[0]$ corresponds to an isotropic line and thus a cusp.  We regroup the sum  in $\Xi(0, \mu, v)$ in terms of $\Gamma$-equivalent cusp classes $[P_r]$, where $r \in \Q$ or $\infty$. Let $\ell_r = \kzxz {r} {-r^2} {1} {-r}$  be an associated  isotropic vector for a rational number $r$ and recall $\ell_\infty= \kzxz {0} {1} {0} {0} $.
\begin{equation}\label{zeroxi}
\Xi(0, \mu, v) =\sum_{[P_r]} \sum_{ 0 \ne w \in L_\mu[0]\cap \Q \ell_r} \Xi(\sqrt v w, z).
\end{equation}
Consider first the sum $[P_\infty]$ part.  Let
$$
L_\mu[0]\cap \Q\ell_\infty = \{w_{m} =\kzxz {0} {\mu_\infty+ m \beta_\infty} {0} {0} \ne 0 :\,  m \in \Z\},
$$
where $\beta_\infty=\beta_{\ell_\infty}$ is the constant  defined at the beginning of this section and $\mu_\infty \in \Q$.  Notice that  two different $w_m$s are  not $\Gamma$-equivalent, and $\Gamma_{w_m}=\Gamma_\infty$. Simple calculation gives
\begin{align*}
&\Xi(\sqrt v w_m, z)
=\sum_{\gamma \in \bar\Gamma_\infty \backslash \bar\Gamma} \beta_1(\frac{\pi  N v|c z+d|^4 (m \beta_\infty +\mu_\infty)^2}{y^2} )
\\
 &=\beta_1(\frac{\pi N v (m\beta_\infty+\mu_\infty)^2 }{y^2} )+\sum_{\gamma \in \bar\Gamma_\infty \backslash \bar\Gamma, c>0} \beta_1(\frac{\pi  N v|c z+d|^4 (m \beta_\infty +\mu_\infty)^2}{y^2} ).
\end{align*}
When $\mu_\infty \notin \beta_\infty \Z$ (i.e., $\mu \notin L$), one has  by Lemma \ref{lem:beta}
\begin{align*}
&\sum_{0 \ne w\in L_\mu[0]\cap \Q \ell_\infty} \Xi(\sqrt v w, z)
\\
&= \sum_{0 \ne m \in \Z}\beta_1(\frac{\pi N v (m\beta_\infty+\mu_\infty)^2}{y^2} ) + e(\mu, z)
\\
&= -\beta_1(\frac{\pi N v \mu_\infty^2}{y^2})  + \frac{2  y}{\beta_\infty\sqrt{Nv}} + f(\frac{\beta_\infty\sqrt{Nv}}{y}, \frac{\mu_\infty}{\beta_\infty})
  + e(\mu, z).
\end{align*}
Here
$$
e(\mu, z) = \sum_{ 0\ne m \in \Z} \sum_{\gamma \in \bar\Gamma_\infty \backslash \bar\Gamma, c>0} \beta_1(\frac{\pi  N v|c z+d|^4 (m \beta_\infty +\mu_\infty)^2}{y^2} ).
$$
Recall that near $t=0$
$$
\beta_1(t) =-\log t + \gamma_1(t)
$$
with
$$
 \gamma_1(t) =\int_1^\infty e^{-y} \frac{dy}y + \int_t^1 \frac{e^{-y}-1}{y} dy.
$$
So we have for $\mu \notin L$ (recall $y =-\frac{\kappa}{2\pi} \log |q_\kappa|$)
\begin{equation}
\sum_{0 \ne w\in L_\mu[0]\cap \Q \ell_\infty} \Xi(\sqrt v w, z)
= -\log|q_\kappa|^2 \frac{\varepsilon_\infty}{\pi \sqrt{Nv}} -2 \log(-\log|q_\kappa|^2) + \psi(\mu, z),
\end{equation}
where
$$
\psi(\mu, z)=-\log\frac{\varepsilon_\infty^2}{4Nv \pi^3} - \gamma_1(\frac{\pi N v \beta_\infty^2}{y^2}) + f(\frac{\beta_\infty\sqrt{Nv}}{y}, \frac{\mu_\infty}{\beta_\infty})+ e(\mu, z).
$$
It is easy to see that  every term  in the sum of $e(\mu, z)$ is uniformly of exponential decay (with respect to $c, d, m \in \Z, c>0, m\ne 0$) as $y$ goes to infinity. So $e(\mu, z)$ is of exponential decay as  $y$ goes to infinity. This implies
\begin{equation}
\lim_{y \rightarrow \infty} \psi(\mu, z) =-\log\frac{\varepsilon_\infty^2}{4Nv \pi^3} -\gamma_1(0) + 2 \sum_{n=1}^\infty \frac{\cos(\frac{2 \pi n \mu_\infty}{\beta_\infty})}n.
\end{equation}

For $\mu\in L$ (i.e., $\mu=0$ in  $L^\sharp/L$), one has
\begin{align*}
\sum_{0 \ne w \in L[0]\cap \Q \ell_\infty} \Xi(\sqrt v w, z)
&= \frac{2y}{\beta_\infty \sqrt{Nv}} +2\log\frac{\sqrt{Nv}\beta_\infty}{y}+ f(\frac{\sqrt{Nv} \beta_\infty}{y}) + e(0, z)
\\
 &=-\frac{\varepsilon_\infty}{2 \pi \sqrt{Nv}}\log|q_{\ell_{\infty}}|^2 -2 \log(-\log|q_{\ell_{\infty}}|^2)+ \psi(0, z),
\end{align*}
with
$$
\psi(0, z)=2\log\frac{4 \pi \sqrt{Nv}} {\varepsilon_\infty}+ f(\frac{\sqrt{Nv} \beta_\infty}{y}) + e(0, z).
$$
So
one has
\begin{equation} \label{eq7.5}
\lim_{y \rightarrow \infty} \psi(0, z) = 2\log\frac{4 \pi \sqrt{Nv}}{\varepsilon_\infty} + f(0),
\end{equation}
as $e(0, z)$ is of exponential decay as $y$ goes to the infinity.

Now look at the sum of $[P_r]$ part,  where $P_r$ is  not $\Gamma$-equivalent to $P_\infty$. This implies that there is no $\gamma =\abcd \in  \Gamma$ such that $\gamma(\infty) =\frac{a}c =r$.
For $w =m\kzxz {r} {-r^2} {1} {-r} \in L_\mu[0] \cap \Q \ell_r$ so that ${\tilde w} = S^{-1}\cdot w=m\kzxz {1} {-r} {-r} {r^2}$.  For $\gamma =\abcd \in \Gamma$, one has
$$
R(w, \gamma z) = \frac{1}2 (w, w(\gamma z))^2 =\frac{N m^2}2  \frac{|(a-rc)z + (b-rd)|^4}{y^2} \sim  \frac{m^2 N}2 (a-rc)^4 y^2
$$
as $y \rightarrow \infty$ as  $a-rc \ne 0$ for all $\gamma \in  \Gamma$. So
$$
\Xi(\sqrt v w, z) =\sum_{\gamma \in \bar\Gamma_{w} \backslash \bar\Gamma} \beta_1(2 \pi R(w, \gamma z))
$$
is smooth and of exponential decay  at the cusp $P_\infty$. Putting everything together, we obtain the  result for $\Xi(0, \mu, v) $ at the cusp $P_\infty$ .
This finally proves Theorem \ref{theo:Singularity}.
\end{proof}

\begin{corollary}  \label{cor:GreenFunction} Let the notation and assumption be as in  Theorem \ref{theo:Singularity} and let $D=-4nN$. Then $\Xi(n, \mu, v)$ is a Green function for $Z(n, \mu, v)^{\Naive} $ in the usual Gillet-Soul\'e sense for $n\ne 0$ and  with  (at most) log-log singularity when $n=0$, and
$$
dd^c [\Xi(n, \mu,v)] + \delta_{Z(n, \mu,v)^{\Naive}} = [\omega(n, \mu,v)].
$$
Here $\omega(n, \mu, v)$ is the differential defined in (\ref{eq:Differential})
$$
Z(n, \mu, v)^{\Naive} = \begin{cases}
  Z(n,  \mu) &\ff D <0,
  \\
   \sum_{P_\ell \,  \hbox{ cusps }\, } \,   g(n, \mu, v, P_\ell) P_\ell &\ff D \ge 0 \hbox{ is a square,}
   \\
    0 &\ff \hbox{ otherwise}.
    \end{cases}
$$
\end{corollary}
\begin{proof} Away from the singularity divisor $Z(n, \mu.v)^{\Naive}$, one has  by \cite[Proposition 11.1]{Kucentral}
$$
dd^c \Xi(n, \mu,v)= \omega(n, \mu,v).
$$
Near the cusps, it is given by Theorem \ref{theo:Singularity}, and we leave the detail to the reader following the idea in \cite[Propositon 11.1]{Kucentral}.
\end{proof}

\section{Modular curve $\mathcal X_0(N)$ and the main  theorem} \label{sect:ModularCurve}

From  now on, we focus on the specific lattice $L$ given in Section  \ref{preliminaries} and $\Gamma=\Gamma_0(N)$.
\begin{comment}
$$
L=\{ x=\kzxz {a} {\frac{b}N} {c} {-a}: \,  a, b, c \in \Z\}
$$
with quadratic form $Q(x) =N \det x=-Na^2 - bc$. One has
$$
\Z/2N\Z \cong L^\sharp/L, \quad r \mapsto \mu_r =\kzxz {\frac{r}{2N}} {0} {0} {-\frac{r}{2N}}.
$$
Notice that $\Gamma_0(N)$ is the subgroup of $\SL_2(\Z)$ which preserves $L$ and acts trivially on $L^\sharp/L$.
\end{comment}
 So our modular curve is  $X_0(N) =Y_0(N) \cup S$ the cusp set $S=\{P_{\frac{1}M}: \,  M|N\}$ with $P_{\frac{1}M}$ is the cusp associated to $\frac{1}M$ (as $N$ is square free). Let
$$
\ell_{\frac{1}{M}}=\kzxz {-M}{1}{-M^2}{M}
$$
be an associated isotropic element.
\begin{comment} The following lemma can be easily checked and is left to the reader.
\begin{lemma}  \label{lem2.1} One has the following isomorphism
$$
\Gamma_0(N) \backslash  \H   \rightarrow \Gamma_0(N) \backslash \D,  \quad  z=x+ iy \mapsto w(z) = \frac{1}{\sqrt{N} y} \left(
  \begin{array}{cc}
  -x  & z\overline{z}\\
     -1&x\\
  \end{array}
\right).
$$
 The inverse is induced  by
$$
w=\kzxz {a} {\frac{b}N} {c} {-a} \mapsto z =\frac{2aN + \sqrt D}{2cN},  \quad D= - 4N Q(w).
$$
\end{lemma}
\end{comment}

\subsection{Some numerical results on Kudla  Green functions}\label{numerical}

\begin{lemma}
The Funke constant  for $P_{\frac{1}{M}}$ is $\varepsilon_{\frac{1}M} = N$, independent of the choices of the cusps.
\end{lemma}
\begin{proof}
Take $\sigma_{M}=\kzxz {1}{0}{M}{1}$. Then  $\sigma_{M}\cdot\ell_{\infty}= \ell_{\frac{1}{M}}$, and
$$
\sigma_{M}^{-1}\cdot L \bigcap \Q \ell_{\infty}=\kzxz {0}{\frac{1}{M}\Z}{0}{0}.
$$
 So we have  $\beta_{\frac{1}M}=\frac{1}{M}$. Next,
 We know that
 \begin{equation}
 \sigma_{M}^{-1}\kzxz {1}{x}{0}{1}\sigma_{M}=\kzxz {1+Mx}{x}{-M^{2}x}{1-Mx} \in  \Gamma_{0}(N)
 \end{equation}
if and only if $x \in \frac{N}M \Z$. This implies  $\kappa_{\frac{1}M} =\frac{N}M$. So $\varepsilon_{\frac{1}M} = N$.

\end{proof}

\begin{comment}

\begin{example}
When $N=1$, for $ r=0, 1$ \begin{eqnarray}
&&\mathcal{E}_{L}(\tau,1)_{\mu_{r}}\nonumber\\
&=&\frac{1}{2}\sum_{n\geq0}\deg(Z(n, \mu_{r}))q^{n}
+\frac{\epsilon(\mu_{r})}{4\pi \sqrt{v}}+\sum_{n>0}\sum_{X\in L_{\mu_{r}}[-n^{2}]/\Gamma}\frac{1}{8\pi \sqrt{v}n}\beta(4\pi vn^{2})q^{-n^{2}}\nonumber\\
&=&\sum_{D\equiv -r^{2}(mod 4),D>0}H(D)q^{\frac{D}{4}}+\frac{\delta_{0r}}{4\pi \sqrt{v}}+\sum_{D\equiv r (mod 2),D>0}\frac{1}{4\pi \sqrt{v}}\beta(\pi vD^{2})q^{-\frac{D^{2}}{4}}.\nonumber
\end{eqnarray}
The sum $$E_{\frac{3}{2}}(\tau)= \mathcal{E}_{L}(4\tau,1)_{\mu_{0}}+\mathcal{E}_{L}(4\tau,1)_{\mu_{1}}$$ is the classical Zagier's Eisenstein series.
\end{example}

  When  $D=-4nN >0$ is a square,  $\Xi(n, \mu, v)$ has log singularity at the cusp $P_\ell$ with
$$
\Xi(n, \mu, v) = - g(n, \mu, v, P_\ell) (\log|q_\ell|^2)
    + \hbox{ smooth}
$$
Here  $q_\ell$ is a local parameter at the cusp $P_\ell$,
$$
\alpha_\Gamma(n, \mu, P_\ell)=\sum_{w \in L_\mu[n] \mod \Gamma}  \delta_{w}
$$
where $0\le \delta_w \le 2$ is the number of  isotropic lines $\ell_w \in \Iso(V)$  which is perpendicular to $w$ and belongs to the same cusp as $\ell$, and $$
g(n,\mu, v, P_\ell)= \frac{\sqrt N }{4 \pi  \sqrt D}\beta_{3/2}(\frac{Dv \pi}N) \alpha_\Gamma(n, \mu, P_\ell).
$$
\end{comment}

\begin{lemma} \label{coealpha}
 When  $D=-4nN >0$ is a square and $L_{\mu}[n]\neq \phi$, one has for every cusp $P_\ell$
$$
\alpha_{\Gamma_0(N)}(n, \mu, P_{\ell})=\left\{
                                \begin{array}{ll}
                                   \sqrt{D},  &\ff  2 \mu \notin L, \\
                                 2\sqrt{D}, &\ff 2 \mu \in L.
                                \end{array}
                              \right.
                              $$
\end{lemma}
\begin{proof}  We will drop the subscript $\Gamma_0(N)$ in the proof. We first assume $P_\ell=P_\infty$. Recall
$$
\alpha(n, \mu, P_{\infty})=\sum_{w \in L_\mu[n] \mod \Gamma_0(N)}  \delta_{w},
$$
where $\delta_w$ is the number of the isotropic lines $\Q \ell$ which is perpendicular to $w$ and whose associated cusp is $\Gamma_0(N)$-equivalent to $P_\infty$. By changing $w$ by its $\Gamma_0(N)$-equivalent element if necessary we may and will assume $(w, \ell_\infty)=0$ (for $\delta_w\ne 0$). This implies
$$
w=w(a, b) = \left(
                                                                                                           \begin{array}{cc}
                                                                                                             \frac{a}{2N} & \frac{b}{N} \\
                                                                                                             0 & - \frac{a}{2N} \\
                                                                                                           \end{array}
                                                                                                         \right)
$$
with
\begin{equation} \label{eq7.3}
a^2 =D, \quad a \equiv r \mod  (2N).
\end{equation}
So
\begin{equation} \label{eq7.4}
w(a, b)^\perp  \cap \Iso(V) =\Q\ell_\infty \cup  \Q\ell(a, b), \quad \ell(a, b) = \kzxz{ab}{b^{2} }{-a^{2}}{-ab}.
\end{equation}
On the other hand, it is straightforward to check that $w(a, b_1)$ is $\Gamma_0(N)$-equivalent to $w(a, b_2)$ if and only if $b_1 \equiv b_2 \pmod a$. Therefore, we only need to consider these $w(a, b)$ with $a$ satisfying (\ref{eq7.3}) and $b \pmod a$. There are at most $2 |a|$ of them.

Now divide the proof into two cases: $N \nmid r$ (i.e., $2 \mu_r \notin L$) and $N |r$ (i.e. $2 \mu_r \in L$).

Assume first that $N \nmid r$. Then (\ref{eq7.3}) has a unique solution $a$, and for this $a$, the cusp $P_{-\frac{b}a}= P_{\ell(a, b)}$ is not $\Gamma_0(N)$-equivalent to $P_\infty$.
So $\delta_w =1$ for each $w(a, b)$.
Therefore we have
$$
\alpha(n, \mu, P_{\infty})=|a| =\sqrt D
$$
in this case.

Next we assume $N|r$. In this case (\ref{eq7.3})  has two solutions $a =\sqrt D$ and $-\sqrt D$. One has also $N|a$.  It is not hard to verify via calculation that $w(a,b)$ and $w(-a, b')$ are $\Gamma_0(N)$-equivalent if and only if $a_2=\hbox{gcd}(a, b) =\hbox{gcd}(a, b')$ has the following properties: $a=Na_2 z$ and $b=a_2w$ with $\hbox{gcd}(Nz, w)=1$, and $b'=a_2x$ for some $x$ with $xw-Nyz=1$ for some integer $y$.  Moreover, in such a case,  $b' \pmod a$ is uniquely determined by $b\pmod a$.

Write $a =Na_1$ and $(a, b)=a_2$ with  $b=a_2 w$.

{\bf Subcase 1}:  We first assume $a_2| a_1$. In this case, we can write $a_1=a_2z$ and thus $a=a_2 Nz$ with $(w, Nz) =1$. So $\delta_{w(\pm a, b)}=2$.
On the other hand, $w(\epsilon a, b)$ is $\Gamma_0(N)$-equivalent to $\Gamma(-\epsilon a, bx)$ with $xw- Nyz=1$ for some $x, y \in \Z$. So the four pairs
$(\pm a, b)$ and $(\pm a, bx)$ contribute $4$ to the sum of $\delta_w$.

{\bf Subcase 2}: Next we assume $a_2 \nmid a_1$.  This means $\hbox{gcd} (a_2, N) >1$.  So  the cusp $P_{\frac{b}{\pm a}}= P_{\frac{a_2z}{\pm Na_1}} $ is not $\Gamma_0(N)$-equivalent to the cusp $P_\infty$. This implies $\delta_{w(\pm a, b)} =1$. On the other hand, for such a pair $(\epsilon a, b)$,  $w(\epsilon a, b)$ is not $\Gamma_0(N)$-equivalent to any other $w(\pm a, b')$.

Combining the two subcases, we see that
$$
\alpha(n, \mu,P_\infty) =2|a|
$$
in this case. This proves the lemma for the cusp $P_\infty$.

Next, we show that  $\alpha(n, \mu, P_{\frac{1}M})$ does not depend on the cusp $P_{\frac{1}M}$ in the following sense.
\begin{equation} \label{eq:Atkin-Lehner}
\alpha(n, \mu, P_{\frac{1}{M}})=\alpha(n, W_{Q}\mu W_{Q} ^{-1}, P_{\infty}),
\end{equation}
where $Q=\frac{M}N$, and $W_Q$ is the associated Atkin-Lehner involution defined as follows.
Since $(M, Q)=1$,  there exist $ \alpha, \beta \in \Z$ with $\alpha Q-M\beta=1$,  so  $\kzxz {1}{\beta}{M}{Q\alpha} \in \Gamma_{0}(M)$. Let
$$W_{Q}= \kzxz {1}{\beta}{M}{Q\alpha} \kzxz{Q}{0}{0}{1}= \kzxz{Q}{\beta}{N}{Q\alpha}$$
 be the associated
 Atkin-Lehner operator.  Then  one has
 $$W_{Q} \Gamma_{0}(N)W_{Q} ^{-1}=\Gamma_{0}(N).
 $$
  It is easy to verify
 $$
 W_{Q} L_\mu[n] W_{Q} ^{-1}=  L_{W_{Q}\mu W_{Q} ^{-1}}[n], \quad W_{Q}\ell_{\infty} W_{Q}^{-1}= \kzxz{Q\alpha}{-\beta}{M\alpha}{-M\beta}=\ell'.
 $$
Notice that $P_{\ell'}=P_{\frac{1}M}$.
So there is a bijective map
\begin{align*}
 L_\mu[n] \bigcap \ell^{\prime , \perp} &\longleftrightarrow    L_{W_{Q}\mu W_{Q} ^{-1}}[n] \bigcap \ell_{\infty}^{\bot},
 \\
 w &\longleftrightarrow  W_{Q}^{-1} wW_{Q}.
 \end{align*}
This proves (\ref{eq:Atkin-Lehner}), and thus the lemma.

\end{proof}

Now we can refine Theorem \ref{theo:Singularity} and Corollary \ref{cor:GreenFunction}  as

\begin{theorem} \label{theo:SpecialGreen}  Let the notation and assumption be as above  and let $D=-4nN$. Then $\Xi(n, \mu, v)$ is a Green function for $Z(n, \mu, v)^{\Naive} $ with (at most) log-log singularity, and
$$
dd^c [\Xi(n, \mu,v)] + \delta_{Z(n, \mu,v)^{\Naive}} = [\omega(n, \mu,v)].
$$
Here $\omega(n, \mu, v)$ is the differential defined in (\ref{eq:Differential})
$$
Z(n, \mu, v)^{\Naive} = \begin{cases}
  Z(n,  \mu) &\ff D <0,
  \\
   g(n, \mu, v)\sum_{M|N } \,   \mathcal  P_{\frac{1}M}  &\ff D \ge 0 \hbox{ is a square},
   \\
    0 &\ff \hbox{ otherwise},
    \end{cases}
$$
and
$$
g(n, \mu, v) =\begin{cases}
 \frac{\sqrt N}{4 \pi  \sqrt{v}}\beta_{3/2}(-4 nv \pi)   &\ff n \ne 0, \mu \notin  \frac{1}2 L/L,
 \\
  \frac{\sqrt N}{2 \pi  \sqrt{v}}\beta_{3/2}(-4 nv \pi)   &\ff n \ne 0, \mu \in  \frac{1}2 L/L,
 \\
   \frac{\sqrt N}{2 \pi \sqrt{v}}  &\ff n=0,  \mu =0,
 \\
  0 &\ff n=0,  \mu \ne 0.
 \end{cases}
$$
Moreover,   for every $M|N$,
\begin{enumerate}
\item  when $D$ is not a square, the Green function $\Xi(n, \mu, v)$ is of exponential decay near cusp $P_{\frac{1}M}$.

\item When  $D=-4N n >0$ is a square, one has
$$
\Xi(n, \mu, v) = - g(n, \mu, v) (\log|q_M|^2)
    -2\psi_M(n, \mu, v; q_M),
$$
where $q_M$ is a local parameter at $P_{\frac{1}M}$, and $\psi_M(n, \mu, v; q_M)$ is of exponential decay near $P_M$. Here $P_{\frac{1}{N}}=P_\infty$, and $\psi_N=\psi_\infty$.

\item When  $D=0$, $\Xi(0, \mu, v)=0$ and $\mu \notin L$, then
\begin{eqnarray}
\Xi(0, 0, v)&=& - g(0, 0, v ) (\log|q_M|^2)-2 \log(-\log|q_M|^2)
    \nonumber\\
& &-2 \psi_M(0, \mu, v;q_M),\nonumber
\end{eqnarray}
 and
 $$
\lim_{|q_M| \rightarrow 0} \psi_M(0, 0, v; q_M) =
   \log\frac{\sqrt N}{4 \pi \sqrt{v}} -\frac{1}2 f(0).
$$
Here $f(0) = \gamma -\log(4\pi)$ is defined in Lemma \ref{lem:beta}.

\end{enumerate}

\end{theorem}

\subsection{Integral model}

Following \cite{KM}, let $\mathcal{Y}_{0}(N)$ $ ( \mathcal{X}_{0}(N) )$ be the moduli stack over $\Z$ of cyclic isogenies of degree $N$ of elliptic
curves (generalized elliptic curves) $\pi : E\rightarrow E^{\prime}$ such that $\ker \pi$ meets every irreducible component of each geometric fiber.
The stack $\mathcal{X}_{0}(N)$ is regular and proper flat over $\Z$ such that $\mathcal{X}_{0}(N)(\C)=X_{0}(N)$ as $N$ is square free. It is a DM-stack. For convenience, we count each point $x$  with multiplicity $\frac{2}{|\Aut (x)|}$ instead of $\frac{1}{|\Aut(x)|}$. It is regular over $\Z$ and smooth over $\Z[\frac{1}N]$. When $p|N$, the special fiber  $\mathcal X_0(N) \pmod p$ has two irreducible components $\mathcal X_p^\infty$ and $\mathcal X_p^0$. Both of them are isomorphic to $\mathcal X_0(N/p) \mod p$, and they intersect at supersingular points. We require $\mathcal X_p^\infty$ to contain the cusp $\mathcal P_\infty \pmod p$ and $\mathcal X_p^0$ to contain the cusp $\mathcal P_0 \pmod p$.     Here for each divisor $Q|N$, let $\mathcal P_{\frac{Q}N}$ be the boundary arithmetic curve associated to the cusp $P_{\frac{Q}N}$, which is the Zariski closure of $P_{\frac{Q}N}$ in $\mathcal X_0(N)$ and has a nice moduli interpretation too. We refer to \cite{Co} for detail. It is known  that $\mathcal P_{\frac{Q}N} \mod p$ lies in $\mathcal X_p^\infty$ (resp. $\mathcal X_p^0$) if and only if $p\nmid Q$ (resp. $p|Q$).

For $r \in \Z/2N$, $\mu_r = \diag(r/2N, -r/2N) \in L^\sharp/L$ and a positive rational number $n \in Q(\mu_r)+\Z$, let $D=-4Nn \equiv r^2 \mod 4N$, $\kay_D=\Q(\sqrt D)$ and the order $\OO_D=\Z[\frac{D+\sqrt D}2]$ of discriminant $D$. When $D<0$,  let $\mathcal Z(n, \mu_r)$ be the Zariski closure of $Z(n, \mu_r)$ in $\mathcal X_0(N)$.  When $D$ is a fundamental discriminant, it has the following  moduli interpretation.

 Let $\mathfrak {n}=[N, \frac{r+\sqrt D}2]$, which is an ideal of $\OO_D$ with norm $N$.  Following \cite{GZ} and \cite{BY}, let $\mathcal Z(n, \mu_r)$ be the moduli stack over $\OO_D$ of the pairs $(x, \iota)$, where $x=(\pi: E \rightarrow E') \in \mathcal Y_0(N)$ and
$$
\iota: \OO_D \hookrightarrow \End(x)=\{ \alpha \in  \End(E) : \pi \alpha \pi^{-1} \in \End(E^{\prime}) \}
$$
is a CM action of $\OO_D$ on $x$ satisfying $\iota( \mathfrak {n}) \ker \pi =0$. It actually descends to a DM stack over $\Z$. It is smooth of dimension $1$. According to a private note Sanrakan shared with us, the same moduli problem for a general $D<0$
also produces a flat, horizontal, and regular stack which is the Zariski closure $\mathcal Z(n, \mu)$---although we don't need this result in this paper.

The forgetful map $$\mathcal{Z}(n , \mu) \rightarrow \mathcal{X}_{0}(N)$$  $$  (\pi : E\rightarrow E^{\prime}, \iota)\rightarrow (\pi : E\rightarrow
E^{\prime}) $$
is a finite   and close  map, which is generically 2 to 1.

\subsection{The metrized Hodge bundle }
Let $\omega_N$ be the Hodge bundle on $\mathcal X_0(N)$ (see \cite{KM}). Then there is a canonical isomorphism $\omega_N^2 \cong \Omega_{\mathcal X_0(N)/\Z}(-S)$, which is also canonically isomorphic to the line bundle of modular forms of weight $2$ for $\Gamma_0(N)$. Here $S$ is the set of cusps. For a positive integer $N$, let $\mathcal M_k(N)$ be the line bundle of weight $k$ with the normalized Petersson  metric
$$
\| f(z)\| = |f(z) (4 \pi e^{-C} y)^{\frac{k}2}|
$$
as defined in (\ref{eq:Petersson}).  This gives a metrized line bundle $\widehat{\mathcal M}_k(N)$ and also induces a metric on $\omega_N$ so that the associated metrized line bundle $\widehat{\omega}_N$ satisfies $\widehat{\omega}_N^k \cong \widehat{\mathcal M}_k(N)$.  From now on, we denote
\newcommand{\rr}{r} %think of a better letter to replace r
\begin{equation} \label{eq:kr}
k=12 \varphi(N), \quad  \rr=N\prod_{p|N} (1+p^{-1})=[\SL_2(\Z): \Gamma_0(N)] =\frac{3}\pi \vol(X_0(N), \mu(z)).
\end{equation}
Recall  that $\Delta_N(z)$ and $\Delta_N^0(z)$  are both  rational sections of $\mathcal M_k(N)$.

\begin{lemma}\label{divisorlemma}
\begin{equation}
\Div\Delta_{N}=\frac{rk}{12}\mathcal{P}_{\infty}-k \sum_{p|N}\frac{p}{p-1}\mathcal{X}_{p}^{0}
\end{equation}
and
\begin{equation}
\Div\Delta_{N}^{0}=\frac{rk}{12}\mathcal{P}_{0}-\frac{k}2\sum_{p \mid N}\frac{p+1}{p-1}\mathcal{X}_{p}^{\infty}-\frac{k}2\sum_{p \mid N}\mathcal{X}_{p}^{0}.
\end{equation}
Here $r$ and $k$ are given by  (\ref{eq:kr}).
\end{lemma}

\begin{proof} Since
\begin{eqnarray}
\Delta_{N}\mid {\kzxz{0}{-1}{1}{0}}(z)=N^{-6\varphi(N)}\Delta_{N}^{0}(\frac{z}{N})=N^{-12\varphi(N)}\Pi_{t\mid N}t^{12a(\frac{N}{t})}\Delta(\frac{tz}{N})^{a(\frac{N}{t})},\nonumber
\end{eqnarray}
we have
$$
\hbox{Div}\Delta_{N}=\frac{rk}{12}\mathcal{P}_{\infty}+\sum_{p \mid N}(-12\varphi(N)+12\sum_{M\mid \frac{N}{p}}a(M))\mathcal{X}_{p}^{0}.
$$
  One has  by (\ref{Deltapower})
\begin{equation}
\sum_{M\mid \frac{N}{p}}a(M)= -\varphi (\frac{N}{p}),
\end{equation}
so
$$
\Div\Delta_{N}=\frac{rk}{12}\mathcal{P}_{\infty}+\sum_{p \mid N}(-12\varphi(N)-12 \varphi (\frac{N}{p}) )\mathcal{X}_{p}^{0}.
$$
Notice
 $\varphi(\frac{N}p) =\frac{1}{p-1}\varphi(N) $,  one has
 $$
 \Div\Delta_{N}=\frac{rk}{12}\mathcal{P}_{\infty} -k \sum_{p|N}\frac{p}{p-1}\mathcal{X}_{p}^{0}
 $$
 as claimed. The second identity follows the same way and is left to the reader.
 \end{proof}

\begin{comment}
\begin{proof} Since
\begin{eqnarray}
\Delta_{N}\mid {\kzxz{0}{-1}{1}{0}}(z)=N^{-6\varphi(N)}\Delta_{N}^{0}(\frac{z}{N})=N^{-12\varphi(N)}\Pi_{t\mid N}t^{12a(\frac{N}{t})}\Delta(\frac{tz}{N})^{a(\frac{N}{t})},\nonumber
\end{eqnarray}
we have
$$
\hbox{Div}\Delta_{N}=\frac{rk}{12}\mathcal{P}_{\infty}+\sum_{p \mid N}(-12\varphi(N)+12\sum_{M\mid \frac{N}{p}}a(M))\mathcal{X}_{p}^{0}.
$$
  One has  by (\ref{Deltapower})
\begin{equation}\sum_{M\mid \frac{N}{p}}a(M)= -\varphi (\frac{N}{p}),\end{equation}
so
$$
\Div\Delta_{N}=\frac{rk}{12}\mathcal{P}_{\infty}+\sum_{p \mid N}(-12\varphi(N)-12 \varphi (\frac{N}{p}) )\mathcal{X}_{p}^{0}.
$$
Notice
 $\varphi(\frac{N}p) =\frac{1}{p-1}\varphi(N) $,  one has
 $$
 \Div\Delta_{N}=\frac{rk}{12}\mathcal{P}_{\infty} -k \sum_{p|N}\frac{p}{p-1}\mathcal{X}_{p}^{0}
 $$
 as claimed. The second identity follows the same way from Corollary \ref{cor:DeltaN0} and is left to the reader.

Recall  $$\Delta_{N}^{0}(z)=N^{-6\varphi(N)}\Pi_{t\mid N}t^{12a(\frac{N}{t})}\Delta(tz)^{a(\frac{N}{t})}$$  and
$$\Delta_{N}^{0}\mid _{\kzxz{0}{-1}{1}{0}}(z)=N^{-6\varphi(N)}\Delta_{N}(\frac{z}{N}),$$
one has
\begin{eqnarray}
\Div\Delta_{N}^{0}&=&\frac{rk}{12}\mathcal{P}_{0}+\sum_{p \mid N}(-6\varphi(N)+12\sum_{M\mid \frac{N}{p}}a(M))\mathcal{X}_{p}^{\infty}-6\varphi(N)\sum_{p \mid N}\mathcal{X}_{p}^{0}\nonumber\\
&=&\frac{rk}{12}\mathcal{P}_{0}+\sum_{p \mid N}(-6\varphi(N)-12 \varphi (\frac{N}{p}))\mathcal{X}_{p}^{\infty}-6\varphi(N)\sum_{p \mid N}\mathcal{X}_{p}^{0}.\nonumber
\end{eqnarray}
\end{proof}
\end{comment}

The following lemma is clear.

\begin{lemma} \label{lem:Delta_N}  Let   $q_z=e(z) $   be  a local parameter of $X_0(N)$ at the cusp $P_\infty$.

\begin{enumerate}

\item The metrized line bundle $\widehat{\omega}_N^k =\widehat{\mathcal M}_k(N)$ has log singularity along cusps with all $\alpha$-index $\alpha_P=\frac{k}2$ at every cusp $P$. At the cusp $P_\infty$, one has
    $$
    \|\Delta_N(z) \| = (-\log|q_z|^2)^{\frac{k}2}|q_z|^{\frac{\rr}{12} k} \varphi(q_z),
$$
with
$$
\varphi(q_z) = e^{-\frac{kC}2} \prod_{n=1}^\infty |(1- q_z)^{24 C_N(n)}|.
$$

%Similarly for $\Delta_N^0(z)$ at the cusp $P_0$ ({\bf put it here if we need it})

\item  Both $\widehat{\Div}(\Delta_N)=(\Div (\Delta_N), -\log\| \Delta_N(z) \|^2)$ and $\widehat{\Div}(\Delta_N^0)=(\Div (\Delta_N^0), -\log\| \Delta_N^0(z) \|^2) $ are arithmetic  divisors (on  $\mathcal X_0(N)$) associated to $\widehat{\omega}_N^k$ with log-log singularity at cusps.

\end{enumerate}

\end{lemma}

We also consider the arithmetic divisor  on  $\mathcal X_0(N)$:
\begin{equation}
\widehat{\Delta}_N = (\frac{rk}{12} \mathcal P_\infty,  -\log \| \Delta_N(z) \|^2).
\end{equation}
One has
\begin{equation} \label{eq6.10}
\widehat{\Div}(\Delta_N)= \widehat{\Delta}_N  - k \sum_{p \mid N}\frac{p}{p-1} \mathcal{X}_{p}^{0}.
\end{equation}

Define
\begin{equation} \label{eq:ArithmeticDivisor}
\widehat{\mathcal Z}(n, \mu, v)
=\begin{cases}
 \widehat{\mathcal Z}(n, \mu, v)^{\Naive}  -2 \widehat{\omega}_N - \sum_{p|N} \mathcal X_p^0 -(0, \log(\frac{v}{N})) &\ff n=0, \mu=0,
 \\
 \widehat{\mathcal Z}(n, \mu, v)^{\Naive} &\hbox{ otherwise}.
 \end{cases}
\end{equation}
\begin{comment}
Concretely,
\begin{align} \label{eq:KudlaDivisor}
&\widehat{\mathcal Z}(n, \mu, v)
\\
 &=\begin{cases}
 (\mathcal Z(n, \mu), \Xi(n, \mu, v)) &\ff  n>0,
 \\
  (0, \Xi(n, \mu, v)) &\ff n<0,  D\ne\square,
 \\
  ( g(n,\mu, v)\sum_{M|N} \mathcal  P_{\frac{1}M}  ,  \Xi(n, \mu, v)) &\ff n <0,  D=\square,
  \\
  ( g(0, 0, v) \sum_{M \mid N} \mathcal P_{\frac{1}M} ,  \Xi(0, 0, v))  -2 \widehat{\omega}_N - \sum_{p|N} \mathcal X_p^0 -(0, \log(\frac{v}{N})) &\ff n=0, \mu=0,
  \\
  (0, 0)  &\ff n=0,  \mu\ne 0.
  \end{cases} \notag
\end{align}
\end{comment}
The arithmetic generating function ($q=e(\tau)$) in the introduction is defined to be
\begin{equation} \label{eq:KudlaGeneratingFunction}
\widehat{\phi}(\tau) = \sum_{\substack{ n \in \frac{1}{2N}\Z \\
                          \mu \in L^\sharp/L
                          \\
                          Q(\mu) \equiv n \pmod 1}}
                           \widehat{\mathcal Z}(n, \mu, v) q^n e_\mu  \in \widehat{\CH}^1_\R(\mathcal X_0(N), S)\otimes \C[L^\sharp/L][[q, q^{-1}]].
\end{equation}

Replacing $\widehat{\omega}_N$ by the class of arithmetic divisor $\frac{1}k \widehat{\Div}(\Delta_N)$, we can rewrite
$\widehat{\phi}(\tau) = (\phi(\tau),  \Xi_L(\tau, z))$ where
\begin{align}\label{genericinfinite}
\phi(\tau) &= \sum_{n, \mu}  \mathcal Z(n, \mu, v) q^n e_\mu  \quad,  \hbox{ and  }
\\
\Xi_L(\tau, z) &= (\Xi(0, 0, \mu) + \frac{2}k \log\|\Delta_N\|^2 - \log\frac{v}N )e_0 +\sum_{n\ne 0 , \mu} \Xi(n, \mu, v) q^n e_\mu. \notag
\end{align}
%Notice that
%\begin{equation}
%d d^c \Xi_L(\tau, z) = \Theta_L(\tau, z),
%\end{equation}
%which is  a vector valued modular form known by Kudla-Millson.

\begin{proposition} \label{prop:Chow} One has

$$
\widehat{\phi}(\tau) \in £¨\widehat{\CH}^1_\R(\mathcal X_0(N))\otimes \C[L^\sharp/L]£©[[q, q^{-1}]].
$$
\end{proposition}
\begin{proof} By Theorem \ref{theo:SpecialGreen}, it suffices to check the case for $\widehat{\mathcal Z}(0,0, v)$.
Notice that $\Delta(\tau)$ is a section of $\omega_N^{12}$. So we have by Theorem \ref{theo:SpecialGreen}
$$
\widehat{\mathcal Z}(0,0, v)=(\mathcal Z, g)
$$
with
\begin{align*}
\mathcal Z&= -\frac{\sqrt N}{2 \pi \sqrt v} \sum_{M|N} \mathcal P_{\frac{1}M} -\frac{1}6 \Div \Delta  -\sum_{p|N} \mathcal X_p^0,
\\
g  &= \Xi(0, 0, v) + \frac{1}6 \log\|\Delta\|^2 -\log\frac{v}N.
\end{align*}
For each $M|N$, choose $\sigma_M \in \SL_2(\Z)$ such that $\sigma_M(\infty) =\frac{1}M$. Then Theorem \ref{theo:SpecialGreen}(3) asserts
\begin{align*}
\Xi(0,0, v)(\sigma_M(z)) &=- \frac{\sqrt N}{2 \pi \sqrt v} (\log|q_M|^2)-2 \log(-\log|q_M|^2) + \hbox{smooth}
\\
 &= - \frac{\sqrt N}{2 \pi \sqrt v} (\log|q|^2)-2 \log(-\log|q|^2) + \hbox{smooth},
\end{align*}
where $q_M= q^{\frac{M}N}$, as the width of the cusp $P_{\frac{N}M}$ is $\frac{N}M$. On the other hand,
$$
\log\|\Delta(\sigma_M(z))\|^2 =\log\|\Delta(z)\|^2=\log (|q|^2) + 12 \log (-\log|q|^2) +\hbox{smooth}.
$$
So we know
$$
g(\sigma_M(z)) = (- \frac{\sqrt N}{2 \pi \sqrt v} +\frac{1}6) \log (|q|^2) + \hbox{smooth}
$$
has just log singularity.
\end{proof}

We first record the following proposition, which is clear by (\ref{eq:VertInfiniteIntersection}) and Corollary  \ref{cor:Value}.

\begin{proposition} \label{prop:value} Let the notation be as above, then
\begin{eqnarray}
&&\deg \widehat{\phi}(\tau)=\sum_{n, \mu}\deg(\widehat{\mathcal Z}(n, \mu, v))  q^{n}e_{\mu} = \langle \widehat{\phi}(\tau),  a(2) \rangle  =\frac{2}{\varphi(N)}\mathcal E_L(\tau, 1).\nonumber
\end{eqnarray}
In general, for  $a(f) =(0, f) \in \widehat{\CH}_\R^1(\mathcal X_0(N), S)$, we have
$$
\langle \widehat{\phi}(\tau), a(f) \rangle  =\frac{1}2 \int_{X_0(N)} f(z) \Theta_L(\tau, z)= \frac{1}2 I(\tau, f)
$$
is  a vector valued modular form valued in $S_L$ for $\Gamma^{\prime}$ of weight $3/2$ and representation $\rho_L$.

%In general, for any $a(f) =(0, f) \in \widehat{\CH}_\R^1(\mathcal X_0(N), S)$ with $S$ being the set of cusps,
%$\langle \widehat{\phi}(\tau), a(f) \rangle $ is a vector valued modular form valued in $S_L$ for $\Gamma^{\prime}$ of weight $3/2$ and representation $\rho_L$.
\end{proposition}

\begin{proposition}  \label{prop:vertical} For every prime $p|N$, one has
$$
\langle \widehat{\phi}(\tau), \mathcal X_p^0 \rangle =\langle \widehat{\phi}(\tau), \mathcal X_p^{\infty} \rangle =\frac{1}{\varphi(N)}\mathcal E_L(\tau, 1) \log p.
$$
\end{proposition}
\begin{proof}
Since
$$
R(w, w_N z) = R(w_N^{-1}\cdot w, z),
$$
and $w_N\cdot L_\mu = L_{-\mu}$, one has by definition
$$
w_N^* \Xi(n, \mu, v) = \Xi(n, -\mu, v) =\Xi(n, \mu, v).
$$
This  implies
$$w_N^* Z(n, \mu, v)^{\Naive}= Z(n, \mu, v)^{\Naive}$$
 on the generic fiber. Since the divisors $\mathcal Z(n, \mu, v)^{\Naive}$ are all horizontal (flat closure of $Z(n, \mu, v)^{\Naive}$), we have
 $$w_N^* \mathcal Z(n, \mu, v)^{\Naive}= \mathcal Z(n, \mu, v)^{\Naive}.$$
One has also $w_N^* \widehat{\Delta}_N = \widehat{\Delta}_N^0$ and  $ w_N^* \mathcal X_p^0= \mathcal X_p^\infty$. Direct calculation using Lemma \ref{divisorlemma} then shows
$$
 w_N^*  \widehat{\mathcal Z}(0, 0, v) = \widehat{\mathcal Z}(0, 0, v),
 $$
 and  so
\begin{comment}
 The same argument also shows that $w_N^* \widehat{\Delta}_N = \widehat{\Delta}_N^0$.  Since $ w_N^* \mathcal X_p^0= \mathcal X_p^\infty$,  Lemma \ref{divisorlemma} implies
 \begin{align*}
 w_N^* (2 \widehat{\omega}_N + \sum_{p | N} \mathcal X_p^0)
 &= w_N^*( \frac{2}k \widehat{\Delta}_N^0 -\sum_{p|N} \frac{p+1}{p-1} \mathcal X_p^\infty)
 \\
 &= ( \frac{2}k \widehat{\Delta}_N -\sum_{p|N} \frac{p+1}{p-1} \mathcal X_p^0)
 \\
 &=  (2 \widehat{\omega}_N + \sum_{p | N} \mathcal X_p^0).
 \end{align*}
 So
 $$
 w_N^*  \widehat{\mathcal Z}(0, 0, v) = \widehat{\mathcal Z}(0, 0, v),
 $$
and
\end{comment}
$$
w_N^* (\widehat{\phi}(\tau)) = \widehat{\phi}(\tau).
$$
Since $w_N$ is an isomorphism, we have
\begin{align*}
&\langle \widehat{\phi}(\tau), \mathcal X_p^0 \rangle
 =\langle \widehat{\phi}(\tau), \mathcal X_p^\infty \rangle
 \\
 &=\frac{1}2 \langle\widehat{\phi}(\tau), \mathcal X_p \rangle
  =\frac{1}2 \langle \widehat{\phi}(\tau),(0, \log p^2)  \rangle
  \\
  &=\frac{1}2 \deg  \widehat{\phi}(\tau) \log p
  = \frac{1}{\varphi(N)}\mathcal E_L(\tau, 1) \log p .
\end{align*}
Here we have used the fact that the principal arithmetic divisor $\widehat{\Div}(p) = (\mathcal X_p, -\log p^2)$. This proves the proposition.
\end{proof}

{\bf Proof of Theorem \ref{maintheo}}:  Now Theorem \ref{maintheo} follows from  Propositions \ref{prop:value} and \ref{prop:vertical}, equation (\ref{eq6.10}), and the following theorem, which will be proved in next section.

\begin{theorem} \label{theo:horizontal} Let the notation be above. Then

$$
\langle \widehat{\phi}(\tau), \widehat{\Delta}_N \rangle_{GS} =  12 \mathcal E_L'(\tau, 1).
$$
\end{theorem}

\section{The proof  of Theorem \ref{theo:horizontal} }  \label{sect:Proof}

%Recall that $N >0$ is assumed to be square free.

\subsection{Some preparation}

\begin{lemma} Two different cusps of $X_0(N)$ reduce to two different cusps modulo $p$ for every prime number $p$. So $\langle \mathcal P_{\frac{1}{M_1}}, \mathcal P_{\frac{1}{M_2}} \rangle =0$ if $M_1 \not\equiv M_2 \pmod N$.
\end{lemma}
\begin{proof} We only need to consider primes $p|N$. If $p$ divides exactly one of the $M_1$ and $M_2$, the two cusps landed in two different branches of $\mathcal X_p$ and thus do not coincide. When $p$ divides both of them,  their reductions $\bar{\mathcal P}_{\frac{1}{M_j}}$  both landed in $\mathcal X_p^0$. On the other hand, $\mathcal X_p^0 $ is isomorphic to the reduction of $\mathcal X_0(N/p)$, under which cusps correspond to cusps. Counting the number of cusps, we see that different cusps which landed in $\mathcal X_p^0$ are still different in the reduction. This proves the lemma.

\end{proof}
\begin{comment}
{\bf your formula is for $p\nmid 2, 3$. What about general formula in Stack setting? The number could be fractions.}
When $ p\mid N,$ and $p \neq 2, 3$, from the work \cite[Appendix, section 1]{Mazur}, we know that $\mathcal{X}_{p}^{0}$ and $\mathcal{X}_{p}^{\infty}$ intersect transversely at all
supersingular  points. So intersection number $<\mathcal{X}_{p}^{0} , \mathcal{X}_{p}^{\infty}>$ is the number of all supersingular  points with multiplicities.
A point $x=(E, H)$ is called supersingular point of $\mathcal{X}_{p}(\frac{N}{p})$ if $E$ is a supersingular elliptic curve and  $H \subseteq E[\frac{N}{p}]$ is a cyclic subgroup with order $\frac{N}{p}$. Notice that $\{ \pm 1\} \subseteq Aut(x)$.

    We just recall parts of results \cite[Appendix Theorem 1.1]{Mazur} and  \cite[Appendix proposition 1.2]{Mazur}.
 At supersingular point $x$, $\mathcal{X}_{0}(N)$ has a singularity whose strict localization is  isomorphic to
$$W(\overline{F}_{p})[[X, Y]]/(X.Y-p^{k})$$
Here $k= \frac{1}{2}\mid Aut(x) \mid,$ which is denoted by $k_{x}.$

Assume $N=p\Pi_{i=1}^{t}p_{i}$.  From this Appendix, we know that the number of supersingular points is
\begin{equation}\label{numbersupersingular}
\begin{cases}
r\frac{p-1}{12(p+1)}-2^{t}(\frac{u}{2}+\frac{v}{3}) &\ff  k=1,
 \\
  2^{t}u &\ff k=2,
 \\
  2^{t}v &\ff k=3,
  \end{cases}
\end{equation}
for some $u, v \in \{ 0, 1\}.$
This could see from the picture of page 177, Theorem 1.1 and proposition 1.2 in  \cite[Appendix]{Mazur}.
\end{comment}

\begin{lemma} \label{lem:VerticalIntersection} One has for each $p|N$,
$$
\langle \mathcal X_p^\infty,  \mathcal X_p^0 \rangle
=-\langle \mathcal X_p^0,  \mathcal X_p^0  \rangle
=-\langle \mathcal X_p^\infty,  \mathcal X_p^\infty  \rangle  =\frac{r(p-1)}{12(p+1)}\log p.
$$
\end{lemma}
\begin{proof} Recall that $\mathcal X_p^\infty$ and $\mathcal X_p^0$ are both isomorphic to the special file $\mathcal X_0(\frac{N}p)_p= \mathcal X_0(\frac{N}p) \pmod p$ and that they intersect properly exactly at all the supersingular points. So
\begin{align*}
\langle \mathcal X_p^\infty,  \mathcal X_p^0 \rangle
&= \sum_{\substack{x \in \mathcal X_0(\frac{N}p)_p(\bar{\mathbb F}_p) \\ \hbox{supersingular}}}   \frac{2}{\vert Aut(x)\vert}
\\
&=[\SL_2(\Z): \Gamma_0(\frac{N}p)] \sum_{\substack{x \in \mathcal X_0(1)_p(\bar{\mathbb F}_p) \\ \hbox{supersingular}}}   \frac{2}{\vert Aut(x)\vert}
\\
&=\frac{r}{p+1} \sum_{\substack{x \in \mathcal X_0(1)_p(\bar{\mathbb F}_p) \\ \hbox{supersingular}}}   \frac{2}{\vert Aut(x)\vert}.
\end{align*}
It is well-known (see for example \cite[corollary 12.4.6]{KM}) that
\begin{equation}\label{supersingularsum}
\frac{p-1}{24}=\sum_{j \in \overline{F}_{p}, E_{j}\hbox{supersingular}}\frac{1}{\vert Aut(E_{j})\vert}.
\end{equation}
So
$$
\langle \mathcal X_p^\infty,  \mathcal X_p^0 \rangle = \frac{r(p-1)}{12 (p+1)}.
$$
 On the other hand
 $$
 \langle \mathcal X_p^{\infty},  \mathcal X_p^{\infty}  \rangle+\langle \mathcal X_p^\infty,  \mathcal X_p^0 \rangle
 = \langle \mathcal X_p^\infty,  (0, \log p^2) \rangle=0.
 $$
So
$$
\langle \mathcal X_p^{\infty},  \mathcal X_p^{\infty}  \rangle=-\langle \mathcal X_p^\infty,  \mathcal X_p^0 \rangle .
$$
\end{proof}

\begin{lemma}  \label{lem:SelfIntersection} One has
\begin{enumerate}
\item
$$
\langle \widehat{\omega}_N, \widehat{\omega}_N \rangle= r(\frac{\zeta(-1)}{2}+\zeta^{\prime}(-1))+\frac{r}{12}C,
$$
where $C=\frac{\log(4\pi) + \gamma}2$ is the normalization constant  in (\ref{eq:Petersson})

\item

$$
\langle \widehat{\Delta}_N,  \widehat{\Delta}_N \rangle =k^{2} r(\frac{\zeta(-1)}{2}+\zeta^{\prime}(-1))+\frac{k^{2}rC}{12}+\frac{k^{2}r}{12}\sum_{p \mid N}\frac{p^{2}}{p^{2}-1}\log p.
$$

\end{enumerate}

\end{lemma}
\begin{proof}  Let $\widehat{\omega}_{N, \Pet}^k$ be the Hodge bundle with the Petersson  metric (via its isomorphism to $\mathcal M_k(N)$)
$$
\|f(z)\|_{\Pet} = |f(z)(4 \pi  y)^{\frac{k}2} | = \|f(z)\| e^{\frac{kC}2}.
$$
According to \cite[Theorem 6.1]{Kuhn2}, we have
\begin{equation}
\langle \widehat{\omega}_{\Pet},\widehat{\omega}_{\Pet}\rangle= r(\frac{\zeta(-1)}{2}+\zeta^{\prime}(-1)).
\end{equation}
So
\begin{align*}
\langle \widehat{\omega}_N,\widehat{\omega}_N\rangle
&=\langle \widehat{\omega}_{\Pet},\widehat{\omega}_{\Pet}\rangle + 2 \langle \widehat{\omega}_{\Pet}, (0, C) \rangle
\\
&=r(\frac{\zeta(-1)}{2}+\zeta^{\prime}(-1))+\deg(\widehat{\omega}_{\Pet})C\\
&=r(\frac{\zeta(-1)}{2}+\zeta^{\prime}(-1))+\frac{r}{12}C
\end{align*}
as claimed.

Next, one has
\begin{align*}
\langle \widehat{\Delta}_N,  \widehat{\Delta}_N \rangle
&= \langle \widehat{\Delta}_N ,  \widehat{\Div}(\Delta_N) \rangle + k\sum_{p|N}\frac{p}{p-1} \langle  \widehat{\Delta}_N,  \mathcal{X}_{p}^{0} \rangle
\\
&=\langle \widehat{\Div}(\Delta_N) ,  \widehat{\Div}(\Delta_N) \rangle + k \sum_{p|N}\frac{p}{p-1} \langle  \mathcal{X}_{p}^{0}, \widehat{\Div}(\Delta_N)  \rangle
\\
&= k^2 \langle \widehat{\omega}_N,\widehat{\omega}_N\rangle - k^2 \sum_{p|N}(\frac{p}{p-1})^2 \langle  \mathcal{X}_{p}^{0}, \mathcal X_p^0  \rangle
\\
&=k^{2} r(\frac{\zeta(-1)}{2}+\zeta^{\prime}(-1))+\frac{k^{2}rC}{12}+\frac{k^{2}r}{12}\sum_{p \mid N}\frac{p^{2}}{p^{2}-1}\log p,
\end{align*}
by Lemma \ref{lem:VerticalIntersection}
\end{proof}

\begin{remark}  \label{rem:SelfIntersection} We remark that Lemma \ref{lem:SelfIntersection} can be proved directly using our explicit description of sections of $\omega_N^k$ without using \cite[Theorem 6.1]{Kuhn2}. Indeed, one has
$$
\langle \widehat{\omega}_N^k , \widehat{\omega}_N^k \rangle =\langle \widehat{\Div}(\Delta_N), \widehat{\Div}(\Delta_N^0 )\rangle.
$$
Now direct calculation gives the lemma. We leave the detail to the reader.
\end{remark}

\subsection{Proof of Theorem  \ref{theo:horizontal}}

In this section, we prove Theorem \ref{theo:horizontal}, which  amounts to check term by term on their Fourier coefficients.

By Theorem \ref{derivatives} and (\ref{eq:ThetaCoefficient}), it suffices to prove
\begin{equation} \label{eq:MainIdentity}
\langle \widehat{\mathcal Z}(n, \mu, v), \widehat{\Delta}_N \rangle   =
\begin{cases}
 -\int_{X_0(N)} \log\|\Delta_N(z)\| \omega(n,\mu, v)  &\ff n \ne 0,
  \\
   - \int_{X_0(N)} \log\|\Delta_N(z)\| ( \omega(0, 0, v)-\frac{dxdy}{2\pi y^{2}})  &\ff n =0, \mu=0.
\end{cases}
\end{equation}

The case $n=0, \mu \ne 0$ is trivial as both sides are zero.

We divide the proof into three cases: $D$ is not a square, $D>0$ is a square, and $D=0$.

{\bf Case 1}:  We first assume that $D$ is not a square. In this case, $\mathcal Z(n, \mu, v)$ and $\mathcal P_\infty$ has no intersection at all. By Proposition \ref{prop:Intersection} and Theorem \ref{theo:SpecialGreen}, one has

$$
\langle \widehat{\mathcal Z}(n, \mu, v),  \widehat{\Delta}_N \rangle
=-\int_{X_0(N)} \log\|\Delta_N\| \omega(n, \mu, v).
$$
This proves the case that $D$ is  not a square.

{\bf Case 2}:  Now we assume that $D$ is a square.  This case is complicated due to self-intersection at $\mathcal P_\infty$. We work out the case $D=0$ and leave the similar (and slightly easier) case $D>0$ to the reader.
Let
$$
\widehat{\mathcal Z}_{1}(0, 0, v)= \widehat{\mathcal Z}(0, 0, v)^{Naive}-\frac{12 g(0, 0, v)}{rk} \widehat{\Delta}_N
=(\mathcal Z_1(0, 0, v), \Xi_1(0, 0, v)).
$$
Then
\begin{eqnarray}\label{intercoeff}
\langle  \widehat{\mathcal Z}(0, 0, v)^{Naive},  \widehat{\Delta}_N \rangle=\langle \widehat{\mathcal Z}_{1}(0, 0, v),   \widehat{\Delta}_N \rangle +\frac{12 g(0, 0, v)}{rk} \langle  \widehat{\Delta}_N,   \widehat{\Delta}_N \rangle.\nonumber
\end{eqnarray}
 We have
\begin{eqnarray}\label{square}
&&\langle \widehat{\mathcal Z}_{1}(0, 0, v),   \widehat{\Delta}_N \rangle\nonumber
\\
&&=\sum_{0<M|N, M<N}\frac{rk}{12}  g(n, \mu, v) \langle \mathcal P_{\frac{1}M}, \mathcal P_\infty \rangle
  +\frac{rk}{12} ( \alpha_{\mathcal Z_1, P_\infty} - \psi_{1,\infty}(0, 0, v, 0))
\nonumber\\
 &&\qquad - \lim_{\epsilon \rightarrow 0} \left( \frac{rk}{12} \alpha_{\mathcal Z_1, P_\infty} \log(- \log \epsilon^{2}) - \frac{1}2\int_{X_0(N)_\epsilon} -\log\| \Delta_N \|^2 \omega_1 \right),\nonumber
\end{eqnarray}
where
$$
\omega_1 =\omega(0, 0, v) - \frac{12g(0, 0, v)}r  \frac{dx dy}{4 \pi y^2}
$$
and
$$
\alpha_{\mathcal Z_1, P_\infty}=1-\frac{6}r g(0, 0, v).
$$
So the limit is equal to
\begin{align*}
&\frac{rk}{12}\alpha_{\mathcal Z_1, P_\infty}\lim_{ \epsilon \rightarrow 0} \left( \log(-\log \epsilon^2) + \frac{12}{r k} \int_{X_0(N)_\epsilon} \log\| \Delta_N\|^2 \frac{dx dy}{4 \pi y^2}  \right)
\\
&+ \lim_{\epsilon\rightarrow 0} \int_{X_0(N)_{\epsilon}} \log\| \Delta_N\| (\omega(0, 0, v)-\frac{dxdy}{2\pi y^{2}})
\\
 &=\frac{rk}{12}\alpha_{\mathcal Z_1, P_\infty}\lim_{ \epsilon \rightarrow 0} \left( \log(-\log \epsilon^2) + \frac{12}{r k} \int_{X_0(N)_\epsilon} \log\| \Delta_N\|^2 \frac{dx dy}{4 \pi y^2}  \right)
 \\
 &+ \int_{X_0(N)} \log\| \Delta_N\| (\omega(0,0, v)-\frac{dxdy}{2\pi y^{2}}).
\end{align*}
Recall (\cite[Lemma 2.8]{Kuhn} that
\begin{equation}\label{kuhnlimit}
\lim_{\epsilon \rightarrow 0} \left( \log(-\log \epsilon^2) + \frac{12}{r k} \int_{X_0(N)_\epsilon} \log\| \Delta_N\|^2 \frac{dx dy}{4 \pi y^2}\right) =\frac{r \pi }{3} C_{0}+ 2 \log (4 \pi) -C.
\end{equation}
Here $C_0$  is the scattering constant given in  Lemma  \ref{scatteringconst}, and $C$ is the normalization  constant in Petersson norm.
Combining this with  Corollary \ref{scatteringconst}, we obtain

\begin{align*}
&\langle \widehat{\mathcal Z}_{1}(0, 0, v),   \widehat{\Delta}_N \rangle
\\
&= \frac{rk}{12}(1-\frac{6}{r}g(0, 0, v) )\bigg(12\zeta(-1) +24\zeta'(-1)
 +C +2\sum_{p|N} \frac{p^2}{p^2-1} \log p\bigg)
\\
&-\frac{rk}{12} \psi_{1,\infty}(0, 0, v, 0)-\int_{X_0(N)} \log\| \Delta_N\| (\omega(0,0, v)-\frac{dxdy}{2\pi y^{2}}).
\end{align*}

Here we recall $\zeta(-1) =-\frac{1}{12}$. On  the other hand,
Theorem \ref{theo:Singularity} implies
\begin{align*}
\psi_{1,\infty}(0, 0, v, 0) &= \lim_{y\rightarrow \infty}(\psi_\infty(0, 0, v, q_z)- \frac{12}{rk} g(0, 0, v)\log\phi(q_z))
\\
&=-\frac{1}{2}\log(\frac{v}{N})-(1-\frac{6}{r}g(0, 0, v))C.
\end{align*}
Therefore, one has by Lemma \ref{lem:SelfIntersection}
 \begin{eqnarray}
&&\langle  \widehat{\mathcal Z}(0, 0, v)^{Naive},  \widehat{\Delta}_N \rangle  \nonumber\\
&=&\langle \widehat{\mathcal{Z}_{1}}(0, 0, v),   \widehat{\Delta}_N \rangle +\frac{12 g(0, 0, v)}{rk} \langle  \widehat{\Delta}_N,   \widehat{\Delta}_N \rangle \nonumber\\
&=&\frac{rk}{24}\log(\frac{v}{N}) +\frac{2}{k}\langle  \widehat{\Delta}_N,   \widehat{\Delta}_N\rangle
-\int_{X_0(N)} \log\| \Delta_N\| (\omega(0,0, v)-\frac{dxdy}{2\pi y^{2}}),\nonumber
\end{eqnarray}

and
\begin{align*}
\langle  \widehat{\mathcal Z}(0, 0, v),  \widehat{\Delta}_N \rangle
&= \langle  \widehat{\mathcal Z}(0, 0, v)^{\Naive},  \widehat{\Delta}_N \rangle -\frac{2}{k} \langle \widehat{\Delta}_N,  \widehat{\Delta}_N\rangle
 \\
 &\quad  -\sum_{p|N} \frac{p+1}{p-1} \langle \mathcal X_p^0, \widehat{\Delta}_N\rangle -\langle (0, \log (\frac{v}{N})),\widehat{\Delta}_N\rangle
  \\
  &=-\int_{X_0(N)} \log\| \Delta_N\| (\omega(0,0, v)-\frac{dxdy}{2\pi y^{2}}).
\end{align*}

This proves the case $D=0$.

\section{Modularity of the arithmetic theta function}\label{sec:modularlarity}

 In this section, we will prove the modularity of $\widehat{\phi}(\tau)$. To simplify the notation, we denote in this section  $X=X_0(N)$ and $\mathcal X=\mathcal X_0(N)$, and let  $S$ be   the set of cusps of $X$.  Let $g_{\GS}$ be a Gillet-Soul\'e Green function for the divisor $\Div \Delta_N$(without log-log singularity), and let $\widehat{\Delta}_{\GS} =(\Div \Delta_N, g_{\GS}) \in \widehat{\CH}_\R^1(\mathcal X)$, and
$
f_N= g_{\GS} + \log\|\Delta_N\|^2$. Then $ a(f_N) =(0, f_N) \in  \widehat{\CH}_\R^1(\mathcal X, S)$ and
$$
\widehat{\Delta}_{\GS}= \widehat{\Div}({\Delta}_N) + a(f_N).
$$
 Theorem \ref{theo:horizontal}  and Proposition \ref{prop:value}   imply the following proposition immediately.

\begin{proposition}  \label{prop8.1} The Gillet-Soul\'e height pairing
$\langle \widehat{\phi}, \widehat{\Delta}_{\GS} \rangle $ is a vector valued modular form of $\Gamma'$  valued in  $\C[L^\sharp/L]$ of weight $3/2$ and representation $\rho_L$.
\end{proposition}

Now we are ready to prove Theorem \ref{theo:modularity} following the idea in   \cite[Chapter 4]{KRYBook} with $\widehat{\omega}$  replaced by $\widehat{\Delta}_{\GS}$.  Let $\mu_{GS} =c_1(\widehat{\Delta}_{\GS})$, $A(X)$ be the space of smooth  functions $f$ on $X$ which are conjugation invariant ($Frob_\infty$-invariant), and let   $A^0(X)$ be the subspace of functions $f \in A(X)$ with
$$
\int_{X} f \mu_{GS}=0.
$$
 For each $p|N$, let
$\mathcal Y_p =\mathcal X_p^\infty- p \mathcal X_p^0$, then $\langle \mathcal Y_p, \widehat{\Delta}_{\GS} \rangle =0$. Let
$\mathcal Y_p^\vee =\frac{1}{\langle \mathcal Y_p, \mathcal Y_p\rangle} \mathcal Y_p$. Finally let $\widetilde{\MW}$ be the orthogonal complement of
$\R \widehat{\Delta}_{\GS}  +\sum_{p|N} \R \mathcal Y_p^\vee + \R a(1) + a(A^0(X))$ in $\widehat{\CH}_\R^1(\mathcal X)$. Then one has

\begin{proposition} \label{prop8.2} (\cite[Propositions 4.1.2, 4.1.4]{KRYBook})
$$
\widehat{\CH}_\R^1(\mathcal X) = \widetilde{\MW}\oplus (\R \widehat{\Delta}_{\GS}  +\sum_{p|N} \R \mathcal Y_p^\vee  + \R a(1)) \oplus a(A^0(X)).
$$
More precisely, every $\widehat{Z} =(\mathcal Z, g_Z)$ decomposes into
$$
\widehat{Z}=\widetilde{Z}_{MW} + \frac{\deg \widehat{Z}}{\deg \widehat{\Delta}_{\GS}} \widehat{\Delta}_{\GS} + \sum_{p|N} \langle \widehat{Z}, \mathcal Y_p\rangle \mathcal Y_p^\vee + 2 \kappa(\widehat{Z}) a(1) + a(f_{\widehat Z})
$$
for some $f_{\widehat Z} \in A^0(X)$, where
$$
 \kappa(\widehat{Z})\deg \widehat{\Delta}_{\GS} =\langle \widehat{Z}, \widehat{\Delta}_{\GS} \rangle
 -\frac{\deg \widehat{Z}}{\deg \widehat{\Delta}_{\GS}} \langle \widehat{\Delta}_{\GS}, \widehat{\Delta}_{\GS} \rangle.
$$
\end{proposition}

\begin{proposition} \label{prop8.3} (\cite[ Remark 4.1.3]{KRYBook}) Let $\MW=J_0(N)\otimes_\Z \R$. Then $\widehat{\mathcal Z}=(\mathcal Z, g_Z)\mapsto  Z  \in \MW$ induces an isomorphism
$$
\widetilde{\MW} \cong \MW,
$$
where $Z$ is the generic fiber of $\mathcal Z$.
The inverse map  is given as follows. Given a rational divisor $Z \in J_0(N)$, let $g_Z$ be the unique harmonic Green function for $Z$ such that

\begin{align*}
d_z d_z^c g_Z -\delta_{Z}&=\frac{\deg(Z)}{\deg \widehat{\Delta}_{\GS}}\mu_{GS},
\\
\int_{X} g_Z\mu_{GS} &=0.
\end{align*}
Let $\mathcal Z$ be a divisor of $\mathcal X$ with rational coefficients such that its generic fiber is $Z$, and it is orthogonal to every irreducible components of  the closed fiber $\mathcal X_p$ for each prime  $p$.  Finally let
$$
\widetilde{\mathcal Z} = \widehat{\mathcal Z} -2 a(\langle \widehat{\Delta}_{\GS}, \widehat{\mathcal Z} \rangle), \quad \widehat{\mathcal Z} =(\mathcal Z, g_Z).
$$
Then the map $Z \mapsto \widetilde{\mathcal Z}$ is the inverse isomorphism.
\end{proposition}

Finally, let $\Delta_z$ be the Laplacian operator with respect to $\mu_{\GS}$. Then  the space $A^0(X)$  has an orthonormal basis $\{ f_j\}$ with
$$
\Delta_z f_j + \lambda_j f_j=0,  \quad  \langle f_i, f_j \rangle = \delta_{ij}, \quad \hbox{ and } 0 < \lambda_1 < \lambda_2 < \cdots,
$$
where  the inner product is given by
$$
\langle f, g \rangle = \int_{X_0(N)} f \bar g  \mu_{\GS}.
$$
In particular, every $f \in  A^0(X)$ has the decomposition
\begin{equation} \label{eq:SpectralDecomposition}
f(z) = \sum  \langle f, f_j \rangle   f_j.
\end{equation}
Recall also (\cite[(4.1.36)]{KRYBook} that
\begin{equation}
d_z  d_z^c f =\Delta_z(f)  \mu_{\GS}.
\end{equation}

With the above preparation, we are now ready to restate Theorem  \ref{theo:modularity} in a slightly  more precise form as follows.

\begin{theorem} Let the notation be as above. Then
$$
\widehat{\phi}(\tau) = \tilde{\phi}_{\MW}(\tau) + \phi_{\GS}(\tau) \widehat{\Delta}_{\GS} + \sum_{p|N} \phi_p(\tau) \mathcal Y_p^\vee + \phi_1(\tau) a(1) + a(\phi_{SM})
$$
where  $\phi_p$, $\phi_1$, and $\phi_{\GS}$  are real analytic   modular forms of $\Gamma'$  of weight $3/2$ and representation  $\rho_L$ valued in $\C[L^\sharp/L]$,  $\tilde{\phi}_{\MW}(\tau)$ is a modular form of $\Gamma'$  of weight $3/2$ and representation  $\rho_L$ valued in  finite dimension vector space $\widetilde{MW} \otimes \C[L^\sharp/L]$, and $\phi_{SM}$ is a modular form of $\Gamma'$  of weight $3/2$ and representation  $\rho_L$ valued in $A^0(X_0(N)) \otimes \C[L^\sharp/L]$.
\end{theorem}
\begin{proof} Under the isomorphism in Proposition \ref{prop8.3}, $\widetilde{\phi}_{MW}$ becomes (here we use Manin's well-known result that the divisor of degree $0$ supported on cusps is torsion and is thus zero in $\CH_\R^1(X)$)
$$
\phi(\tau)_\Q - \frac{\deg \widehat{\phi}} {\deg \widehat{\Delta}_{\GS}} \Div (\Delta_N)_\Q
  =\sum_{n >0,  \mu} (Z(n,\mu) - \deg Z(n, \mu)  P_\infty) q_{\tau}^n e_\mu,
$$
which is modular by either the main result of Gross-Kohnen-Zagier \cite{GKZ} (note that Jacobi forms there are the same as vector valued modular forms we used here), or Borcherds' modularity result for $\phi(\tau)_\Q$ (see \cite[Theorem 4.5.1]{KRYBook}) and Proposition \ref{prop:value}.  Next,
$$
\phi_{\GS}(\tau) \langle \widehat{\Delta}_{\GS}, a(1) \rangle = \langle \widehat{\phi}, a(1) \rangle
$$
implies that $\phi_{\GS}(\tau)$ is modular by Proposition \ref{prop:value}. For a given $p|N$, 
$
\phi_p(\tau)  =\langle \widehat{\phi}, \mathcal Y_p \rangle   
$
 is modular by Proposition  \ref{prop:vertical}.  The identity
$$
\langle \widehat{\phi}, \widehat{\Delta}_{\GS} \rangle = \phi_{\GS} \langle \widehat{\Delta}_{\GS},\widehat{\Delta}_{\GS} \rangle 
+ \phi_1(\tau) \langle  a(1) ,  \widehat{\Delta}_{\GS} \rangle
$$ 
 implies that $\phi_1(\tau)$ is modular by Proposition \ref{prop8.1}.

 Finally, we have by   Proposition \ref{prop8.2},
 \begin{equation}\label{Greenfundecom}
 \phi_{SM} = \Xi_{L}(\tau, z) - g_{MW}  -\phi_{\GS}(\tau) g_{\GS} - \phi_1(\tau) \in A^0(X),
 \end{equation}
Since 
\begin{equation}\label{Kudlacurrent}
d_z d_z^c \Xi_{L}(\tau, z)+\sum_{n, \mu}\delta_{Z(n, \mu)}q^{n}e_{\mu}=\Theta_L(\tau, z) \end{equation} and
\begin{equation}\label{current}
d_z d_z^c (g_{MW}+ \phi_{\GS}(\tau)g_{\GS})+\sum_{n, \mu}\delta_{Z(n, \mu)}q^{n}e_{\mu}=  \frac{\deg\widehat{\phi}}{\deg \widehat{\Delta}_{GS}} \mu_{GS},
\end{equation}
 we have 
 $$
 d_z d_z^c \phi_{SM} = \Theta_L(\tau, z)-\frac{\deg\widehat{\phi}}{\deg \widehat{\Delta}_{GS}} \mu_{GS} .
 $$
 Write 
 $$
 \phi_{SM}(\tau, z) = \sum  \langle \phi_{SM}, f_j \rangle f_j(z)
 $$
 as in (\ref{eq:SpectralDecomposition}). Then 
 \begin{align*}
 \langle \phi_{SM} , f_j \rangle &=-\frac{1}{\lambda_j} \int_{X_0(N)} \phi_{SM}(\tau, z) \Delta_z(\bar{f}_j) \mu_{\GS}
 \\
 &= -\frac{1}{\lambda_j} \int_{X_0(N)} \phi_{SM}(\tau, z) d_z d_z^c \bar{f}_j
 \\
 &=-\frac{1}{\lambda_j} \int_{X_0(N)} d_z d_z^c \phi_{SM}(\tau, z)  \bar{f}_j
 \\
 &=\int_{X_0(N)} (\Theta_L(\tau, z)-\frac{\deg\widehat{\phi}}{\deg \widehat{\Delta}_{GS}} \mu_{GS} )\bar{f_j} 
 \\
 &= \int_{X_0(N)} \Theta_L(\tau, z)\bar{f_j},
\end{align*}
 which is modular.
    Therefore $\phi_{SM}(\tau, z) $ is modular as a function of $\tau$.
\end{proof}
\begin{comment}
From the proof of above theorem, one obtains

{\bf Proof of Proposition \ref{prop:value}.}
For any $f \in A^{0}(X)$, one has
\begin{eqnarray}
\langle \widehat{\phi}(\tau), a(f) \rangle & =&\frac{1}2 \int_{\X_0(N)} f(z) d_z d_z^c \Xi_L(\tau, z)\nonumber\\
&=& \frac{1}2 \int_{\X_0(N)} f(z) d_z d_z^c(g_{MW}  +\phi_{\GS}(\tau) g_{\GS} + \phi_{SM}) \nonumber
\end{eqnarray}
 is modular since $g_{MW}$, $\phi_{\GS}(\tau)$ and $\phi_{SM}$ are modular forms. When $f=2$, one has
$$
\deg \widehat{\phi} = \int_{\X_0(N)} \Theta_L(\tau, z)  =I(\tau,1) =\frac{2}{\varphi(N)} \mathcal E_L(\tau, 1)
$$
by Corollary \ref{cor:Value}.

\end{comment}

This proof was explained to one of us (T.Y.) by Sid Sankaran and avoids the metrized line bundle with log singularity and Eisenstein series in the usual spectral decomposition  of $A^0(X)$ with respect to $\mu(z) =\frac{dx dy}{ y^2} =4\pi  c_1(\widehat{\omega}_N)$. We thank him for kindly sharing his idea and other help.

\end{document}